\documentclass[12pt]{amsart}
\usepackage{amsmath,amssymb,amsfonts,amsthm,amsopn}
\usepackage{graphics}
\usepackage{mathrsfs}
\usepackage{graphicx,tikz}
\usepackage[all]{xy}
\usepackage{amsmath}
\usepackage{multirow, longtable, makecell, caption, array}
\usepackage{amssymb}
\usepackage{mathtools}
\usepackage{pb-diagram}
\usepackage{cellspace} %

  \def\Spnr{Sp(d,\R)}
  \def\Gltwonr{GL(2d,\R)}

\newcommand{\ft}{Fourier transform}

\newtheorem{tm}{Theorem}[section]
\newtheorem{lemma}[tm]{Lemma}
\newtheorem{prop}[tm]{Proposition}
\newtheorem{cor}[tm]{Corollary}

\newtheorem{theorem}{Theorem}[section]
\newtheorem{corollary}[theorem]{Corollary}
\newtheorem{definition}[theorem]{Definition}
\newtheorem{example}[theorem]{Example}

\newtheorem{proposition}[theorem]{Proposition}
\newtheorem{remark}[theorem]{Remark}

\newcommand{\beqa}{\begin{eqnarray*}}
\newcommand{\eeqa}{\end{eqnarray*}}

\newcommand{\field}[1]{\mathbb{#1}}
\newcommand{\bR}{\field{R}}        
\def\la{\lambda}

 \def\cF{\mathcal{F}}              
 \def\cS{\mathcal{S}}
 \def\cD{\mathcal{D}}
 
 \def\cB{\mathcal{B}}

 \def\cM{\mathcal{M}}

 \def\cA{\mathcal{A}}
 
 \def\cI{\mathcal{I}}
 \def\cC{\mathcal{C}}

\def\a{\aleph}
\def\hf{\hat{f}}
\def\hg{\hat{g}}
\def\bg{\bar{g}}

\def\rd{\bR^d}

\def\rdd{{\bR^{2d}}}

\def\lrd{L^2(\rd)}
\def\lrdd{L^2(\rdd)}

\def\intrd{\int_{\rd}}
\def\intrdd{\int_{\rdd}}

\def\R{\right)}

\def\<{\left<}
\def\>{\right>}

\def\mv1{M_v^1}

\def\mpq{M^{p,q}}
\def\Mmpq{M_m^{p,q}}
\def\phas{(x,\xi )}

\def\c{\hfill\break}

\def\o{\xi}
\def\a{\alpha}

\def\N{\mathbb{N}}
\def\R{\mathbb{R}}
\def\Ren{\mathbb{R}^d}
\def\Renn{\mathbb{R}^{2d}}

\def\Fur{\mathcal{F}}

\def\H{{\mathbb H}}

\def\f{\varphi}

\def\Opw{Op_{w}}
\def\Sn2{S_{2}(L^{2}(\Ren))}
\def\S1{S_{1}(L^{2}(\Ren))}
\def\sig00{\sigma_{0,0}}

\def\la{\langle}
\def\ra{\rangle}

\newcommand{\A}{\mathcal{A}}

\begin{document}

\title[Wigner analysis of
operators]{Wigner analysis of
operators.\\ Part I: pseudodifferential operators\\ and wave fronts}

\author{Elena Cordero and Luigi Rodino}
\address{Department of Mathematics,  University of Torino, Italy}
\address{Dipartimento di Matematica,  University of Torino, Italy}
\email{elena.cordero@unito.it}
\email{luigi.rodino@unito.it}

\date{}

\begin{abstract}
We perform Wigner analysis of linear operators. Namely, the standard time-frequency representation \emph{Short-time Fourier Transform} (STFT) is replaced by the  $\mathcal{A}$-\emph{Wigner distribution} defined by $W_\cA (f)=\mu(\cA)(f\otimes\bar{f})$, where $\cA$ is a $4d\times 4d$ symplectic matrix and $\mu(\cA)$ is an associate metaplectic operator. Basic examples are given by the so-called $\tau$-Wigner distributions.  Such representations provide a new characterization for modulation spaces when $\tau\in (0,1)$.  Furthermore, they can be efficiently  employed in the study of the off-diagonal decay for pseudodifferential operators with symbols in the  Sj\"ostrand class (in particular, in the H\"{o}rmander class 
$S^0_{0,0}$). The novelty relies on defining time-frequency representations via metaplectic operators, developing a conceptual framework and  paving the way for a new understanding  of quantization procedures.   We deduce micro-local properties for pseudodifferential operators in terms of the Wigner wave front set. Finally, we compare the Wigner with the global H\"{o}rmander wave front set  and  identify the possible presence of a ghost region
in  the Wigner wave front. \par
In the second part of the paper applications to Fourier integral operators and Schr\"{o}dinger equations will be given.
\end{abstract}

\maketitle

\section{Introduction}
The Wigner distribution was introduced by E. Wigner in 1932 \cite{Wigner32} in Quantum Mechanics  and fifteen years later employed  by J. Ville  \cite{Ville48} in Signal Analysis.
Since then the Wigner distribution has been applied in many different frameworks by mathematicians, engineers and physicists: it is one of the most popular time-frequency representations, cf. \cite{bogetal,Cohen1,Cohen2}.

\begin{definition} Consider $f,g\in\lrd$. The Wigner distribution $W f$ is defined as \index{Wigner distribution}
	\begin{equation}\label{WD}
	Wf\phas=W(f,f)\phas=\intrd f(x+\frac t2)\overline{f(x-\frac t2)}e^{-2\pi i t\o}\,dt;
	\end{equation}
 the cross-Wigner distribution $W(f,g)$ is
	\begin{equation}\label{CWD}
	W(f,g)\phas=\intrd f(x+\frac t2)\overline{g(x-\frac t2)}e^{-2\pi i t\o}\,dt.
	\end{equation}
\end{definition}
$Wf$ and $W(f,g)$ turn out to be well defined in $\lrdd$, with $\|W(f,g)\|_{\lrdd}=\|f\|_{\lrd}\|g\|_{\lrd}$. A strictly related time-frequency representation is given by the Short-time Fourier Transform (STFT):
\begin{equation}\label{stftdef}
V_g f\phas=\la f,M_\o T_xg\ra=\intrd f(t)\overline{g(t-x)}\,e^{-2\pi i t  \o}dt,\quad x,\o\in\rd,
\end{equation}
where $f$ is in $\lrd$ and the \emph{window function} $g$ is the Gaussian, as in the original definition of Gabor 1946 \cite{Gab46}, or belonging to some space of regular functions.\par 
Advantages and drawbacks of the use of the Wigner transform  with respect to the STFT are well described in the textbook of Gr\"ochenig \cite[Chapter 4]{grochenig} where we read about quadratic representations $G(f,f)$ of Wigner type:\par \emph{`` the non-linearity of these time-frequency representations   makes the numerical treatment of signals difficult and often impractical''} 
and further\par
\emph{`` On the positive side, a genuinely quadratic time-frequency representation of the form $G(f,f)$ does not depend on an auxiliary window $g$. Thus it should display the time-frequency behaviour of $f$ in a pure, unobstructed form.''} 

In the last ten years, time-frequency analysis methods have been applied to the study of the partial differential equations, with the emphasis on pseudodifferential and Fourier integral operator theory. As technical tool, preference was given to the STFT, with numerical applications in terms of Gabor frames. Let us refer to Cordero and Rodino 2020 \cite{Elena-book}, and corresponding bibliography. For a given linear operator $P$, with action  on $\lrd$ or on more general functional spaces, one considers there the STFT kernel $h$, defined as the distributional kernel   of an operator $H$ satisfying 
\begin{equation}\label{I1}
V_g (Pf) =HV_g f,
\end{equation}
that is, with formal integral notation:
\begin{equation}\label{I2}
V_g (Pf)\phas =\intrdd h(x,\xi,y,\eta) V_gf(y,\eta)\,dyd\eta.
\end{equation}
Attention is then fixed on the properties of almost-diagonalization 
for $h$, in the case when $P$ is a pseudodifferential operator, with generalization to Fourier integral operators appearing in the study of Schr\"odinger equations.\par
In our present work, addressing again to linear operators, we abandon the STFT in favour of the Wigner transform. Namely, using the original  Wigner approach \cite{Wigner32}, later  developed by Cohen and  many other authors (see e.g. \cite{Cohen1,Cohen2}),  we consider $K$ such that 
\begin{equation}\label{I3}
W(Pf) = KW(f)
\end{equation}
and its kernel $k$
\begin{equation}\label{I4}
W(Pf)\phas = \intrdd k(x,\xi,y,\eta) Wf(y,\eta)\,dyd\eta.
\end{equation}
Part I of the paper concerns the case of the pseudodifferential operators. The results will be applied in Part II to Schr\"{o}dinger equations and corresponding  propagators. We shall argue in terms of the $\tau$-Wigner representations, $0<\tau<1$, see the definition in the sequel. In these last years they have become popular for their use in Signal Theory and Quantum Mechanics, see for example  \cite{bogetal,CdGDN2018,CdGDN2020,Luef2019,Luef2020}.

We shall begin, in this introduction, to describe a circle of ideas for $Wf, W(f,g)$ in the elementary $\lrd$ setting, general results being left to the next sessions. As for the class of the pseudodifferential operators, 
the most natural choice is given by symbols  in the H\"{o}rmander class 
$S^0_{0,0}(\rdd)$, consisting of smooth functions $a$ on $\rdd$ such that 
\begin{equation}\label{I5}
|\partial_x^\alpha \partial_\xi^\beta a\phas|\leq c_{\alpha,\beta},\quad\alpha,\beta\in\N^d,\quad x,\xi\in\rd.
\end{equation}
The corresponding pseudodifferential operator is defined by  the Weyl quantization 
\begin{equation}\label{I6}
\Opw(a)f(x)=\intrdd e^{2\pi i (x-y)\xi } a\left(\frac{x+y}{2},\xi\right) f(y)\,dyd\xi.
\end{equation}
If $a\in S^0_{0,0}(\rdd)$ then $\Opw(a)$ is bounded on $\lrd$, according to the celebrated result of Calder\'{o}n and  Vaillancourt \cite{calderon-vaillancourt71}. Note that the Wigner distribution can provide an alternative definition of $\Opw(a)$ by the identity 
\begin{equation}\label{I7}
\la \Opw(a)f,g\ra=\la a,W(g,f)\ra,\quad f,g\in\lrd.
\end{equation}
Our preliminary result will be the following. 
\begin{proposition}\label{P1}
	Assume $a\in S^0_{0,0}(\rdd)$. Then  
	\begin{equation}\label{I8}
	W(\Opw(a)f,g)= \Opw(b)W(f,g), 
	\end{equation}
	with $b\in S^0_{0,0}(\bR^{4d})$ given by
	\begin{equation}\label{I9}
b(x,\xi,u,v)= a(x-v/2,\xi+u/2),
	\end{equation}
where $u$ and $v$ are the dual variables of $x$ and $\xi$, respectively. 	
\end{proposition}
Using the notation $a(x,D)$ for $\Opw(a)$, with $D=-i\partial$, we may write
\begin{equation}
\Opw(b)=a(x-\frac1{4\pi}D_\xi,\xi+\frac1{4\pi}D_x),
\end{equation}
 acting on $W(f,g)\phas$. By using the identity $W(g,f)=\overline{W(f,g)}$, we easily deduce:
 \begin{theorem}\label{T1}
 For $a\in S^0_{0,0}(\rdd)$  we have
 \begin{equation}\label{I11}
 W(\Opw(a)f)=KWf
 \end{equation}
 with
  \begin{equation}\label{I12}
K=a (x-\frac 1{4\pi}D_\xi,\xi+\frac1{4\pi}D_x) \bar{a} (x+\frac 1{4\pi}D_\xi,\xi-\frac1{4\pi}D_x).
 \end{equation}
 \end{theorem}
Note that $K$ is a pseudodifferential operator with symbol in $S^0_{0,0}(\bR^{4d})$, hence if $f\in\lrd$ we have $Wf\in\lrdd$ and $KWf\in \lrdd$, consistently with the left-hand side of \eqref{I11}, where $ W(\Opw(a)f)\in\lrdd$.

Proposition \ref{P1} and Theorem \ref{T1} are not new in literature, at least at the formal level. If we fix $a=x_j$, $a=\xi_j$, we obtain respectively, after extension to higher order operators:
\begin{equation}\label{I13}
W(x_j f,g)=(x_j-\frac1{4\pi}D_{\xi_j})W(f,g),
\end{equation}
\begin{equation}\label{I14}
W(D_{x_j} f,g)=(2\pi\xi_j+\frac1{2}D_{x_j})W(f,g),
\end{equation}
and similar formulas for $W(f,x_jg)$, $W(f,D_{x_j}g )$, recapturing the so-called Moyal operators. General identities of the type \eqref{I11}, \eqref{I12} appeared in Quantum Mechanics  and Signal Theory, see the papers of L. Cohen, for example Cohen and Galleani \cite{CG2002}, and the contribution of de Gosson concerning Bopp quantization, cf. Dias, de Gosson, Prata \cite{DdGP2013}.

We may now present the counterpart for the Wigner kernel $k$ in \eqref{I4} of the STFT almost-diagonalization.
\begin{theorem}\label{T2}
Let $a\in S^0_{0,0}(\rdd)$ and let $k$ be the corresponding Wigner kernel, that is the one of the operator $K$ in \eqref{I11}. Write for short $z=\phas$, $w=(y,\eta)$, so that 
\begin{equation}\label{I15}
W(\Opw (a)f)(z)=\intrdd k(z,w)Wf(w)\,dw.
\end{equation}
\end{theorem}
Then, for every integer $N\geq 0$, 
\begin{equation}\label{I16}
k(z,w)\la z-w\ra^N
\end{equation}
 \emph{is the kernel of an operator $K_N$ bounded on $\lrdd$, namely of a pseudodifferential operator $K_N=\Opw(a_N)$ with symbol $a_N\in S^0_{0,0}(\R^{4d})$. In \eqref{I16} we set, as standard,} $\la t\ra=(1+|t|^2)^{1/2}$. \par This off-diagonal algebraic decay looks promising for sparsity property of $K$. However, the computational methods available in literature for Wigner transform seem not adapted  to such applications. We shall then  limit to a qualitative result, based on the following definition. As standard in the study of partial differential equations, we address to high frequencies.
\begin{definition}\label{D1}
Let $f\in\lrd$. We define $WF(f)$, the Wigner wave front set of $f$, as follows: $z_0=(x_0,\xi_0)\notin WF(f)$, $z_0\not=0$, if there exists a  conic open neighbourhood  $\Gamma_{z_0}\subset\rdd$ of $z_0$ such that for every integer $N\geq 0$
\begin{equation}\label{I17}
\int_{\Gamma_{z_0}} |z|^{2N}|Wf(z)|^2\,dz<\infty.
\end{equation} 
\end{definition}
Hence $WF(f)$ is a closed cone in $\rdd\setminus\{0\}$.  Note that $WF(f)$ is the natural version in the Wigner context of the global wave front set  of {H}\"{o}rmander 1989 \cite{hormanderglobalwfs91}.

From Theorem \ref{T2} we may now deduce the following micro-local property for pseudodifferential operators.
\begin{theorem}\label{T3}
	Consider $a\in S^0_{0,0}(\rdd)$. Then for every  $f\in\lrd$:
	\begin{equation}\label{I18bis}
 WF(\Opw(a)f) \subset WF(f).
	\end{equation}
\end{theorem}

Finally, we describe in short the generalizations of the next sections.

After the preliminary Section $2$, where we recall basic facts and known results to be used in the sequel, in Section $3$ we extend the definition of the Wigner distribution by considering $\tau$-Wigner distributions, playing a crucial role in the sequel. 
 \begin{definition}\label{def2.2}
	For  $\tau\in [0,1]$, $f,g\in\lrd$, we define the (cross-)$\tau$-Wigner distribution  by
	\begin{equation}
	\label{tau-Wigner distribution}
	W_{\tau}(f,g)(x,\xi) = \intrd e^{-2\pi it\xi}f(x+\tau t)\overline{g(x-(1-\tau)t)}dt=\cF_2\mathfrak{T}_{\tau}(f\otimes \bg)\phas,
	\end{equation} 
	where $\cF_2$ is  the partial Fourier transform with respect to the second variables $y$
	and the change of coordinates $\mathfrak{T}_{\tau}$ is given by $\mathfrak{T}_{\tau} F(x,y)=F(x+\tau y,x-(1-\tau)y)$.
	For $f=g$ we obtain the $\tau$-Wigner distribution
	$$W_\tau f:=	W_{\tau}(f,f).$$
\end{definition}
$Wf$ is the particular case of $\tau$-Wigner distribution corresponding to the value $\tau = 1/2$. The cases $\tau=0$ and $\tau=1$ correspond to the   (cross-)Rihaczek and conjugate-(cross-)Rihaczek distribution, respectively; they will remain out of consideration in the main part of the statements, because of their peculiarities.
 
 The $\tau$-quantization $Op_\tau(a)$, which we may define by extending \eqref{I7} as 
 \begin{equation}\label{C21}
\la  Op_\tau(a) f,g\ra=\la a, W_\tau(g,f)\ra,\quad f,g\in\lrd,
 \end{equation}
was already in Shubin \cite{shubin} and it is largely used in modern literature. Let us also emphasize the connection with L. Cohen \cite{Cohen1,Cohen2}. We recall that a time-frequency representation belongs to the Cohen class if it is of the form $Qf=Wf\ast\sigma$, where $\sigma$, so-called kernel of $Q$, is a fixed function or element of $\cS'(\rdd)$. The $\tau$-Wigner distribution belongs to the Cohen class, with kernel 
\begin{equation}
\label{kernel of tau-Wigner distribution}
\sigma_\tau \phas= 
\begin{cases}
\frac{2^d}{|2\tau -1|^d}e^{2\pi i\frac{2}{2\tau -1}x \xi} &\tau \neq \frac{1}{2}\\
&\\
\delta & \tau =\frac{1}{2},
\end{cases}
\end{equation}
cf.  Proposition 1.3.27 in \cite{Elena-book}.\par
In Section $3$ we use the $\tau$-Wigner distribution to characterize the function spaces involved in this study: the modulation spaces (see Subsection $2.1$ for their definition and main properties). 

An interesting result for the time-frequency community is that modulation spaces can be defined by replacing the STFT with $\tau$-Wigner distributions.  Namely, given a fixed non-zero window function $g\in\cS(\rd)$, we characterize the modulation spaces  $M^{p,q}_{v_s}(\rd)$, $1\leq p,q\leq \infty$, $s\in\bR$, as the subspace of functions/distributions $f\in \cS'(\rd)$, satisfying for a fixed $\tau\in (0,1)$,
 \begin{equation}\label{C22}
 \|W_\tau (f,g)\|_{L^{p,q}_{v_s}}=\left(\intrd \left(\intrd (|W_\tau (f,g)\phas|^p \la \phas \ra^{ps} dx\right)^\frac qp d\xi \right)^{\frac1q}<\infty,
 \end{equation}
where $\la z\ra=(1+|z|^2)^{1/2}$, with $\|f\|_{M^{p,q}_{v_s}}\asymp  \|W_\tau (f,g)\|_{L^{p,q}_{v_s}}$. For $p\not=2$, the previous characterization does not hold when $\tau=0$ or $\tau=1$. 

A fundamental step to develop our theory is contained in Corollary \ref{C3.14bis} below:

\emph{Assume $\tau\in [0,1]$, $1\leq p\leq \infty$ and $s\geq 0$.
	If    $f,g \in M^p_{v_s}(\rd)$ then 
	$ W_\tau(f,g)\in M^p_{v_s}(\rdd)$ with
	\begin{equation*}
	\|W_\tau(f,g)\|_{M^p_{v_s}}\lesssim \|f\|_{M^p_{v_s}} \|g\|_{M^p_{v_s}},
	\end{equation*} }
\par
When $p=2$ the Hilbert space $M^2_s(\rd)$ is the so-called Shubin space, cf. Remark 4.4.3, $(iii)$ in \cite{Elena-book}, so that its norm can be computed by means of $\tau$-Wigner distributions as well.

 In Section $4$ we generalize further the notion of Wigner distribution. Namely, fixed a $4d\times 4d$ symplectic matrix $\cA\in Sp(2d,\bR)$, we define the (cross-)$\cA$-Wigner distribution of $f,g\in\lrd$ by  
\begin{equation}\label{C24}
W_\cA (f,g)=\mu(\cA) (f\otimes \bar{g}),
\end{equation}
and $W_\cA f=W_\cA (f,f)$, where  $\mu(\cA)$ is a metapectic operator associated with $\cA$ (see Subsection $2.2$ below). In Section $4$ the analysis of $W_\cA f$ is limited to some basic facts, used in the sequel of the paper. In particular, we characterize the $\cA$-Wigner distributions which belong to the Cohen class.  Note that the STFT can be viewed as $\cA$-Wigner distribution, cf. Remark \ref{rem1E}. We believe that this point of view of defining   time-frequency representations via metaplectic operators could find applications in Quantum Mechanics as well as in Quantum Harmonic Analysis (cf. \cite{DdGLP2011}, \cite{DdGP2013}, and the contributions related to  Convolutional Neural Networks  \cite{Monika2018, Monika2021}).

 In this paper, by using $\cA$-Wigner distributions, we prove a general version of Proposition \ref{P1}, expressed in terms of $W_\tau$ and  $\tau$-pseudodifferential operators  in the frame of Schwartz  spaces $\cS,\cS'$. Beside the use of \eqref{C24}, basic tool for the proof is the covariance property
\begin{equation}\label{C25}
\mu(\cA) Op_w(a)=Op_w(a\circ \cA^{-1})\mu(\cA),
\end{equation}
valid for general symbols $a\in\cS'(\rdd)$. \par
A further study of the $\cA$-Wigner distributions would be interesting \emph{per se}, we believe. In Part II we shall apply them to Schr\"odinger equations with quadratic Hamiltonians. In short: starting from the standard $\tau$-Wigner representation of the initial datum, the evolved solution will require a representation in terms of $\cA$-Wigner distributions, and their use will be then natural in related problems of Quantum Mechanics.\par
Section $5$ is devoted to  almost-diagonalization and  wave front set. We begin by proving Proposition \ref{P1} and Theorem \ref{T3} in full generality. Namely, by using the results of Sections $3$ and $4$, we extend the functional frame to the modulation spaces in the context of the $\tau$-Wigner representations. Moreover, symbols of pseudodifferential operators will be considered in the Sj\"ostrand's class, i.e., the modulation space $M^{\infty,1}(\rd)$ with weights, containing $S^0_{0,0}$ as a subspace. We shall then prove a generalized version of Theorem \ref{T2}, by extending the statement to modulation spaces, and Theorem \ref{T3}, by considering the $\tau$-version of Definition \ref{D1}. 

Finally, in Subsection $5.3$ we compare the Wigner with the global H\"{o}rmander wave front
set, and we identify the possible presence of a ghost part
in the Wigner wave front.

Concerning Bibliography, we observe that several references are certainly missing, in particular in the field of Mathematical Physics. To our excuse, we may say that the literature on the Wigner distribution is enormous, and it seems impossible to list even the more relevant contributions.
\section {Preliminaries and Function spaces}
\textbf{Notations.} The \emph{reflection operator}  $\cI$ is given by
$$\cI f(t)= f(-t).$$
Translations, modulations and time-frequency shifts are defined as standard, for $z=\phas$
$$\pi(z)f(t)=M_\xi T_xf(t) =e^{2\pi i \xi t } f(t-x).$$
We also write $f^\ast(t)=\overline{f(-t)}$.
Recall the Fourier transform 
 $${\hat
	{f}}(\o)=\Fur f(\o)=\intrd
f(t)e^{-2\pi i t\o}dt$$
and the symplectic Fourier transform
\begin{equation}\label{C1SFT}
\cF_\sigma a(z)=\intrdd e^{-2\pi i \sigma(z,z')} a(z')\, dz',
\end{equation}
with $\sigma$ the standard symplectic form $\sigma(z,z')=Jz\cdot z'$, 
where the symplectic matrix
$J$ is defined as
$$
J=\begin{pmatrix} 0_{d\times d}&I_{d\times d}\\-I_{d\times d}&0_{d\times d}\end{pmatrix},
$$
(here $I_{d\times d}$, $0_{d\times d}$ are  the $d\times d$ identity matrix and null matrix, respectively).
The Fourier transform and symplectic Fourier transform are related by
\begin{equation}\label{C1SFT-FT}
\cF_\sigma a(z)=\cF a(Jz)=\cF(a\circ J)(z), \quad a\in\cS(\rdd).
\end{equation}

We denote by $GL(2d,\bR)$  the linear group of $2d\times 2d$ invertible matrices; for a complex-valued function $F$ on $\rdd$ and $L\in GL(2d,\bR)$ we define
\begin{equation}\label{Ltransf}
\mathfrak{T}_{L} F(x,y)=\sqrt{|\det L|}F(L(x,y)), \quad (x,y)\in\rdd,
\end{equation}
with the convention 
$$L(x,y)= L \left(\begin{array}{c}
x\\
y
\end{array}\right),\quad (x,y)\in\rdd.$$

\subsection{Modulation
spaces}

We denote by $v$  a continuous, positive,  submultiplicative  weight function on $\rd$, that is, 
$ v(z_1+z_2)\leq v(z_1)v(z_2)$, for all $ z_1,z_2\in\Ren$.
We say that $m\in \mathcal{M}_v(\rd)$ if $m$ is a positive, continuous  weight function  on $\Ren$ {\it
	$v$-moderate}:
$ m(z_1+z_2)\leq Cv(z_1)m(z_2)$  for all $z_1,z_2\in\Ren$.

 We focus on weights on $\rdd$ of the type
\begin{equation}\label{weightvs}
v_s(z)=\la z\ra^s=(1+|z|^2)^{s/2},\quad z\in\rdd,
\end{equation}
and weight functions on $\bR^{4d}$:
\begin{equation}\label{weight1tensorvs}
( v_s\otimes 1)(z,\zeta)=(1+|z|^2)^{s/2},\quad (1\otimes v_s)(z,\zeta)=(1+|\zeta|^2)^{s/2},\quad z,\zeta\in\rdd.
\end{equation}
For $s<0$, $v_s$ is $v_{|s|}$-moderate.\par 
Given two weight functions $m_1,m_2$ on $\rd$, we write $$(m_1\otimes m_2)(x,\o)=m_1(x)m_2(\o),\quad x,\o\in \rd,$$
and similarly for weights $m_1,m_2$ on $\rdd$.

The modulation spaces were introduced by Feichtinger in \cite{feichtinger-modulation}  (see also  Galperin and Samarah   \cite{Galperin2004} for the quasi-Banach setting). They are  now available in many textbooks, see e.g. \cite{KB2020,Elena-book,grochenig}.

	Fix a non-zero window $g\in\cS(\rd)$ and consider  a weight $m\in\mathcal{M}_v$ and indices $1\leq p,q\leq \infty$. The modulation space $M^{p,q}_m(\rd)$ consists of all tempered distributions $f\in\cS'(\rd)$ such that 
\begin{equation}\label{norm-mod}
\|f\|_{M^{p,q}_m}=\|V_gf\|_{L^{p,q}_m}=\left(\intrd\left(\intrd |V_g f \phas|^p m\phas^p dx  \right)^{\frac qp}d\o\right)^\frac1q <\infty
\end{equation}
(obvious modifications with $p=\infty$ or $q=\infty)$. The STFT $V_g f$ was defined in \eqref{stftdef}.
For simplicity, we write $M^p_m(\rd)$ for $M^{p,p}_m(\rd)$ and $M^{p,q}(\rd)$ if $m\equiv 1$.

The space $M^{p,q}(\mathbb{R}^d)$ is a Banach space whose definition is
independent of the choice of the window $g$, in the sense that different
non-zero window functions yield equivalent norms and the window class can be extended to the modulation space $M^1_v(\rd)$ (also known as Feichtinger algebra). The  modulation space $M^{\infty,1}(\rd)$ is also called Sj\"ostrand's class \cite{Sjostrand1}.

Modulation spaces enjoy the following inclusion properties: 
\begin{equation}
\mathcal{S}(\mathbb{R}^{d})\subseteq M^{p_{1},q_{1}}_m(\mathbb{R}%
^{d})\subseteq M^{p_{2},q_{2}}_m(\mathbb{R}^{d})\subseteq \mathcal{S}^{\prime
}(\mathbb{R}^{d}),\quad p_{1}\leq p_{2},\,\,q_{1}\leq q_{2}.
\label{modspaceincl1}
\end{equation}%

The closure of $\mathcal{S}(\mathbb{R}^d)$ in the $M^{p,q}_m$-norm is denoted $%
\mathcal{M}_m^{p,q}(\mathbb{R}^d)$. Then 
\begin{equation*}
\mathcal{M}_m^{p,q}(\mathbb{R}^d) \subseteq M^{p,q}_m(\mathbb{R}^d), \quad {%
	\mbox and }\,\, \mathcal{M}^{p,q}_m (\mathbb{R}^d) = M^{p,q}_m (\mathbb{R}^d),
\end{equation*}
provided $p<\infty$ and $q<\infty$.

For technical purposes we recall the spaces that can be viewed as images under \ft\, of the modulation spaces.  For $p,q\in [1,\infty]$, the Wiener amalgam spaces $W(\mathcal{F}L^{p},L^{q})(\mathbb{R}^{d})$ are
given by the distributions $f\in \mathcal{S}^{\prime }(\mathbb{R}^{d})$ such
that 
\begin{equation*}
\Vert f\Vert _{W(\mathcal{F}L^{p},L^{q})(\mathbb{R}^{d})}:=\left( \int_{%
	\mathbb{R}^{d}}\left( \int_{\mathbb{R}^{d}}|V_{g}f(x,\xi)|^{p}\,d\xi
\right) ^{q/p}dx\right) ^{1/q}<\infty \,
\end{equation*}%
(obvious changes for $p=\infty $ or $q=\infty $). Using Parseval
identity in \eqref{stftdef}, we can write \thinspace\ $V_{g}f(x,\xi )=e^{-2\pi ix\xi
}V_{\hat{g}}\hat{f}(\xi ,-x)$, hence 
\begin{equation*}
|V_{g}f(x,\xi )|=|V_{\hat{g}}\hat{f}(\xi ,-x)|=|\mathcal{F}(\hat{f}%
\,T_{\xi }\overline{\hat{g}})(-x)|
\end{equation*}%
so that 
\begin{equation*}
\Vert f\Vert _{{M}^{p,q}}=\left( \int_{\mathbb{R}^{d}}\Vert \hat{f}\
T_{\xi }\overline{\hat{g}}\Vert _{\mathcal{F}L^{p}}^{q}(\xi )\ d\xi
\right) ^{1/q}=\Vert \hat{f}\Vert _{W(\mathcal{F}L^{p},L^{q})}.
\end{equation*}
This proves the claim that Wiener amalgam spaces are  the image under
Fourier transform\thinspace\ of modulation spaces: 
\begin{equation}
\mathcal{F}({M}^{p,q})=W(\mathcal{F}L^{p},L^{q}).  \label{W-M}
\end{equation}

\subsection{The metaplectic
representation and miscellaneous results} The symplectic group is defined
by
\begin{equation}\label{defsymplectic}
\Spnr=\left\{g\in\Gltwonr:\;^t\!gJg=J\right\}.
\end{equation}
 The metaplectic or Shale-Weil representation
$\mu$ is a unitary
representation of the (double
cover of the) symplectic
group $\Spnr$ on $\lrd$. It  arises as intertwining operator between the
standard Schr\"odinger representation $\rho$ of the Heisenberg
group $\H^d$ and the representation that is obtained from it by
composing $\rho$ with the action of $\Spnr$ by automorphisms on
$\H^d$ (see, e.g., \cite{folland89}).
 For elements of $\Spnr$ in
special form, the metaplectic
representation can be
computed explicitly, with definition up to a phase factor. In particular,
\begin{align}
\mu(J)f&=\cF f\label{muJ}\\
\mu\left(\begin{pmatrix} I_{d\times d}&0_{d\times d}\\
C&I_{d\times d}\end{pmatrix}\right)f(x)
&= e^{i\pi 
	x\cdot C x}f(x)\label{lower}.
\end{align}
\begin{proposition}\label{Folland1} Let
	$\mathcal{A}=\begin{pmatrix}
	A&B\\C&D\end{pmatrix}\in
	Sp(d,\R)$ with  $\det A\not=0$, then
	\begin{equation}\label{f4}
	\mu(\A)f(x)=(\det
	A)^{-1/2}\int e^{-\pi i
		x\cdot CA^{-1}x+2\pi i
		\xi\cdot A^{-1} x+\pi
		i\xi\cdot
		A^{-1}B\xi}\hat{f}(\xi)\,d\xi.
	\end{equation}
\end{proposition}
We shall need the following computation of Fourier transforms.
\begin{proposition}\label{Folland2} Let
	$B$ a real, invertible, symmetric $d\times d$ matrix, and let $F_B(z)= e^{i \pi z\cdot Bz}$. Then the distribution Fourier transform of $F_B$ is given by
	\begin{equation}\label{f5}
\widehat{F_B}(\zeta)= e^{\pi i \sharp (B)/4}|\det B| e^{-\pi i \zeta\cdot B^{-1}\zeta},
	\end{equation}
	where $ \sharp (B)$ is the number of positive eigenvalues of $B$ minus the number of negative eigenvalues.
\end{proposition}
%


In the sequel we shall use  the $2d\times 2d$ matrix 
\begin{equation}\label{A'}
\mathcal{T}_\tau=\left(\begin{array}{cc}
0_{d\times d} & \left(1-\tau\right)I_{d\times d}\\
-\tau I_{d\times d} & 0_{d\times d}
\end{array}\right),
\end{equation}
and its  properties  \cite[Lemma 2.2]{CNT}:
\begin{lemma}\label{atau} For $\tau\in [0,1]$ we have
	$$\left(\mathcal{T}_\tau\right)^{\top}=-\mathcal{T}_{1-\tau}$$
	 and  $$\mathcal{T}_\tau+\mathcal{T}_{1-\tau}=J.
	$$ 
\end{lemma}

%
Let us recall the convolution relation for the STFT (see, e.g., Lemma \cite{Elena-book}).
 \begin{lemma}\label{C1ConvSTFT}
	Consider $g,h,\gamma\in\cS(\rd)\setminus\{0\}$ such that $\la h,\gamma\ra\not=0$ and $f\in\cS'(\rd)$. Then 
	\begin{equation}\label{C1ConvSTFTf}
	|V_g f\phas|\leq \frac{1}{|\la h,\gamma\ra|}(|V_h f|\ast|V_g \gamma|)\phas,\quad \forall \phas\in\rdd.
	\end{equation}
\end{lemma}
\section{Properties of the  $\tau$-Wigner distributions}
The main goal of this section relies on the characterization of modulation spaces by $\tau$-Wigner distributions, when $\tau\in (0,1)$. We note that if $\tau=0$ or $\tau=1$ the characterization does not hold.
\subsection{Preliminaries on $\tau$-Wigner distributions}
	For $\tau\in [0,1]$  consider the matrix 
\begin{equation}\label{A-taumatr}
{L}_{\tau}=\left(\begin{array}{cc}
I_{d\times d} & \tau I_{d\times d}\\
I_{d\times d} & -(1-\tau) I_{d\times d}
\end{array}\right).
\end{equation}
Define the change of coordinates $\mathfrak{T}_{\tau}:=\mathfrak{T}_{L_\tau}$ according to \eqref{Ltransf}
and  $\cF_2$ to be  the partial Fourier transform with respect to the second variables $y$
\begin{equation}\label{FT2}
\cF_2 F(x,\o)=\intrd e^{-2\pi i y\cdot \o}F(x,y)\,dy,\quad F\in L^2(\rdd).
\end{equation}
	For  $\tau\in [0,1]$, the (cross-)$\tau$-Wigner distribution  is defined in \eqref{tau-Wigner distribution}: $W_\tau(f,g)=\cF_2 \mathfrak{T}_{\tau}(f\otimes \bar{g})$.\par
The case $\tau=1/2$ is the cross-Wigner distribution $W(f,g)$.
For $\tau=0$, $W_0 (f,g)$ is the (cross-)Rihaczek distribution (also denoted by $R(f,g)$). In detail,
\begin{equation}\label{C1Rckdistr}
W_0(f,g)\phas=R(f,g)\phas=e^{-2\pi i x \xi}f(x)\overline{\hat{g}(\xi)},\quad f,g\in\lrd.
\end{equation}	
For $\tau=1$, $W_1 (f,g)$ is the  conjugate-(cross-)Rihaczek distribution $R^\ast(f,g)$ 
\begin{equation}\label{C1Rckdistr2}
W_1(f,g)\phas=R^\ast(f,g)\phas=\overline{R(g,f)}=e^{2\pi i x \xi}\overline{g(x)}\hat{f}(\xi),\quad f,g\in\lrd.
\end{equation}		

In terms of the symplectic \ft\, (cf. \eqref{C1SFT-FT})
\begin{equation}\label{FT-Rihaczek-simpl}
\cF_\sigma{V_g f}(z)= W_0(f,g)(z),\quad z\in\rdd, \quad f,g\in\lrd.
\end{equation}
or, in terms of the conjugate-(cross-)-Rihaczek distribution,
\begin{equation}\label{FT-Rihaczek-conj}
\cF_\sigma{V_g f}(z)= \overline{W_1(g,f)(z)},\quad z\in\rdd,\quad f,g\in\lrd.
\end{equation}

In what follows we shall use the following formula for the STFT of the $\tau$-Wigner distribution $ W_\tau(f,g)$ (cf. \cite[Lemma 2.3]{CNT} \cite[Lemma 1.3.38]{Elena-book}):
 \begin{lemma}
	\label{lem:STFT of tauWig}
	Consider $\tau\in\left[0,1\right]$, $\varphi_{1},\varphi_{2}\in\mathcal{S}\left(\mathbb{R}^{d}\right)$,
	$f,g\in\cS(\rd)$ and set $\Phi_{\tau}=W_{\tau}\left(\varphi_{1},\varphi_{2}\right)\in\mathcal{S}\left(\mathbb{R}^{2d}\right)$.
	Then, 
	\[
	{V}_{\Phi_{\tau}}W_\tau(f,g)\left(z,\zeta\right)=e^{-2\pi iz_{2}\zeta_{2}}V_{\varphi_{1}}f\left(z-\mathcal{T}_{1-\tau}\zeta\right)\overline{V_{\varphi_{2}}g\left(z+\mathcal{T}_\tau\zeta\right),}
	\]
	where $z=\left(z_{1},z_{2}\right),\,\zeta=\left(\zeta_{1},\zeta_{2}\right)\in\mathbb{R}^{2d}$ and the matrix $\mathcal{T}_\tau$ is defined in \eqref{A'}.
	In particular, 
	\[
	\left|{V}_{\Phi_{\tau}}W_{\tau}\left(f,g\right)\right|=\left|V_{\varphi_{1}}f\left(z-\mathcal{T}_{1-\tau}\zeta\right)\right|\cdot\left|V_{\varphi_{2}}g\left(z+\mathcal{T}_\tau\zeta\right)\right|.
	\]	
\end{lemma}
We remark that the preceding result can be easily extended to the case of distributions $f,g\in\cS'(\rd)$ by standard approximation arguments.

	\begin{definition}
		For $\tau\in (0,1)$, we define the operator $Q_{\tau}$:
		\begin{equation}\label{C1Atau}
		Q_{\tau}:f(t)\longmapsto \cI {f}\biggl(\frac{1-\tau}{\tau}t\biggr).
		\end{equation}
		In other words, $Q_\tau=D_{\frac{1-\tau}{\tau}}\cI$, where we set $D_{\frac{1-\tau}{\tau}}f(t):=f(\frac{1-\tau}{\tau}t)$.
	\end{definition}

	For $\tau\in(0,1)$ and $f\in L^p(\rd)$, $1\leq p\leq\infty$, we obtain
	\begin{equation}\label{ataulp}
	\|Q_{\tau}f\|_{L^p(\rd)} = \frac{\tau^{\frac{d}{p}}}{(1 -\tau)^{\frac{d}{p}}}\|f\|_{L^p(\rd)},
	\end{equation}
	with the convention $d/\infty=0$.
	We can express the $\tau$-Wigner in terms of the STFT as follows \cite[Prop. 1.3.30]{Elena-book}. We  need the following matrix $\mathcal{B}_{\tau}$ and its inverse:
	\begin{equation}
	\mathcal{B}_{\tau}=\left(\begin{array}{cc}
	\frac{1}{1-\tau}I_{d\times d} & 0_{d\times d}\\
	0_{d\times d} & \frac{1}{\tau}I_{d\times d}
	\end{array}\right)\quad\quad  \mathcal{B}_\tau^{-1}=\begin{pmatrix}(1-\tau) I_d&0_d\\0_d&\tau I_d\end{pmatrix} \label{eq:Btau}.
	\end{equation}
	\begin{proposition}
		For $\tau\in (0,1)$, $f,g\in\lrd$,  and $\mathcal{B}_\tau$ defined in \eqref{eq:Btau}, we have 
		\begin{equation}\label{C1WtauSTFT}
		W_\tau (f,g)(x,\xi)= \frac{1}{\tau^d}e^{2\pi i\frac{1}{\tau}x \xi}V_{Q_{\tau}g}f\left(\frac{1}{1-\tau}x,\frac{1}{\tau}\xi\right)=\frac{1}{\tau^d}e^{2\pi i\frac{1}{\tau}x \xi}V_{Q_{\tau}g}f\left(\mathcal{B}_\tau(x,\xi)\right).
		\end{equation}
	\end{proposition}

\begin{corollary}\label{cor1}
	Under the assumptions of the proposition above,
		\begin{align}
V_g f (x,\xi)&=\tau^d e^{-2\pi i(1- \tau)x \xi}	W_\tau (f,Q_{\tau}^{-1}g)\left((1-\tau)x, \tau\xi\right)\notag\\
&=\tau^d e^{-2\pi i(1- \tau)x \xi}	W_\tau (f,Q_{\tau}^{-1}g)\left(\mathcal{B}^{-1}_\tau\phas\right)\label{C1WtauSTFTcons}
	\end{align}
	where
	\begin{equation}\label{atauinv}
	Q_\tau^{-1}g(t)=\mathcal{I}g\left(\frac{\tau}{1-\tau}t\right),\quad t\in\rd.
	\end{equation}
\end{corollary}
\begin{proof}
	It is a straightforward consequence of formula \eqref{C1WtauSTFT}.
\end{proof}

\begin{lemma}[Convolution inequalities for $\tau$-Wigner distributions]\label{convtauwigner}
	Consider $g,h,\gamma\in\cS(\rd)\setminus\{0\}$ such that $\la h,\gamma\ra\not=0$, $\tau\in (0,1)$, the operator $Q_\tau$ defined in \eqref{C1Atau}. Then, for any   $f\in\cS'(\rd)$ we have
	\begin{equation}\label{convtaueq}
|W_\tau(f,g)\phas|\leq\frac{1}{(1-\tau)^d}\frac1{|\la Q_{\tau} h,\gamma\ra|}|W_\tau(f,h)|\ast | W_\tau(\gamma,g)|\phas,\quad \phas\in\rdd.
	\end{equation} 
\end{lemma}
\begin{proof}
	We use formula \eqref{C1WtauSTFT} and write 
	$$|V_{Q_{\tau}g}f|\left(\mathcal{B}_\tau(x,\xi)\right)=\tau^d	|W_\tau (f,g)|(x,\xi),$$
	so that $ \tau^d	|W_\tau (f,g)|(\mathcal{B}_\tau^{-1}(x,\xi))=|V_{Q_{\tau}g}f|\phas$. Observe that the representations above are well-defined continuous functions on $\rdd$ (cf., e.g., \cite[Corollary 1.2.19]{Elena-book}). The convolution relations for the STFT \eqref{C1ConvSTFTf} let us infer
\begin{equation}\label{stima1}	|W_\tau (f,g)|(\mathcal{B}_\tau^{-1}(x,\xi))\leq \frac{t^d}{|\la Q_{\tau} h,\gamma\ra|}(|W_\tau (f,h)(\mathcal{B}_\tau^{-1}\cdot)|\ast |W_\tau(\gamma,g)(\mathcal{B}_\tau^{-1}\cdot)|)\phas.
\end{equation}
	Using the change of variables $\mathcal{B}^{-1}_\tau(u,v)=(u',v')$, $dudv=\frac{1}{\tau^d(1-\tau)^d}du'dv'$, we work out
	\begin{align*}
|W_\tau (f,h)(\mathcal{B}_\tau^{-1}\cdot)|&\ast |W_\tau(\gamma,g)(\mathcal{B}_\tau^{-1}\cdot)|\phas\\
&=\intrdd |W_\tau (f,h)(\mathcal{B}_\tau^{-1}\phas-\mathcal{B}_\tau^{-1}(u,v))|\,|W_\tau(\gamma,g)(\mathcal{B}_\tau^{-1}(u,v))|\,dudv\\
&=\frac{1}{\tau^d(1-\tau)^d}| W_\tau (f,h)|\ast |W_\tau(\gamma,g)|(\mathcal{B}_\tau^{-1}\phas)
	\end{align*}
	for every $\phas\in\rdd$. Replacing $\mathcal{B}_\tau^{-1}\phas$ by $\phas$ in \eqref{stima1}, we obtain the estimate in \eqref{convtaueq}.
	\end{proof}
\begin{corollary}\label{Cor2} Under the assumptions of the lemma above, for $\phas\in\rdd$
\begin{equation}\label{Wconv2}
	|W_\tau (f,Q_\tau^{-1}g)|\phas\leq \frac{1}{(1-\tau)^d}\frac{1}{|\la h,\gamma\ra|}|W_\tau (f,Q_\tau^{-1}h)|\ast|W_\tau(\gamma, Q_\tau^{-1}g)|\phas.
\end{equation}
\end{corollary}
\begin{proof}
Replace $g$ by $Q_\tau^{-1}g$ and $h$ by $Q_\tau^{-1}h$.
\end{proof}
\subsection{Characterization of modulation spaces via $\tau$-Wigner distributions}
For $\tau\in (0,1)$, we can now provide a new characterization for modulation spaces in terms of the $\tau$-Wigner distribution. We refer to \cite{modspsymprot2020} for other characterizations.
\begin{proposition}\label{charmodWigner}
Fix a window $g$ in $M^1_v(\rd)$ and a weight function $m\in\mathcal{M}_v(\rdd)$. For $\tau\in (0,1)$ define
\begin{equation}\label{weightmtau}
m_\tau \phas= m(\cB_\tau \phas),\quad \phas\in\rdd,\quad v_\tau\phas= v(\cB_\tau \phas).
\end{equation}
For $1\leq p,q\leq\infty$ and $f \in M^{p,q}_m(\rd)$ we have the norm equivalence
\begin{equation}\label{norm-eq-wigner}
\|f\|_{\Mmpq}\asymp\|V_gf\|_{L^{p,q}_m} =\frac{1}{(1-\tau)^d}\|W_\tau (f,Q_\tau^{-1}g)\|_{L^{p,q}_{m_\tau}},
\end{equation}
where the operator $Q_\tau $ is defined in \eqref{C1Atau}. In particular,
\begin{equation}\label{norm-eq-wigner2}
\|f\|_{\Mmpq}\asymp_\tau\|W_\tau (f,Q_\tau^{-1}g)\|_{L^{p,q}_{m_\tau}}.
\end{equation}
\end{proposition}
\begin{proof}
	Let us start with a window $g$ in $\cS(\rd)\setminus\{0\}$. Then the characterization in \eqref{norm-eq-wigner} follows by using the definition of modulation spaces and the Corollary \ref{cor1}.
	In detail,
	$$\|V_g f\|_{L^{p,q}_m}=\tau^d\|W_\tau (f,Q_\tau^{-1}g)(\mathcal{B}_\tau^{-1}\cdot)\|_{L^{p,q}_{m}}=\frac{1}{(1-\tau)^d}\|W_\tau (f,Q_\tau^{-1}g)\|_{L^{p,q}_{m_\tau}},$$
	where we performed the change of variables $\cB^{-1}_\tau\phas=(x',\o')$.\par
	For a more general window $g$ in $M^1_v(\rd)$, the proof follows the same argument as in \cite[Theorem 2.3.12(iii)]{Elena-book}, replacing the convolution inequalities for the STFT with the corresponding ones for the $\tau$-Wigner, cf. Lemma \ref{convtauwigner} and Corollary \ref{Cor2}.
\end{proof}
\begin{corollary}\label{charWigner}
	Under the assumptions of Proposition \ref{charmodWigner},
	\begin{equation}\label{chartauwigner}
	\|f\|_{\Mmpq}\asymp_\tau \|W_\tau(f,g)\|_{L^{p,q}_{m_\tau}}.
	\end{equation}
\end{corollary}
\begin{proof}
	The claim is a straightforward application of the change-window property for the STFT and the $\tau$-Wigner distribution in Lemmas \ref{C1ConvSTFT} and \ref{convtauwigner}, respectively.
\end{proof}

\begin{remark}\label{tau0-1}
(i) For $\tau=0$,  by \eqref{FT-Rihaczek-simpl},  \begin{align*}\|f\|_{\Mmpq}&\asymp \|V_g f\|_{L^{p,q}_m }=\|\cF_\sigma W_0(f,g)\|_{L^{p,q}_m}=\|\cF W_0(f,g)(J\cdot)\|_{L^{p,q}_{m}}\\
&=\|\cF W_0(f,g)\|_{L^{p,q}_{m\circ J^{-1}}}.
\end{align*}
(ii) For $\tau=1$,  by \eqref{C1Rckdistr2},
$$ \|f\|_{\Mmpq}\asymp \|V_g f\|_{L^{p,q}_m }=\|\cF_\sigma W_1(g,f)\|_{L^{p,q}_m}=\|\cF W_1(g,f)\|_{L^{p,q}_{m\circ J^{-1}}}.$$
\end{remark}

Notice that for $\tau=0$ or $\tau=1$ we do not get a characterization similar to the case $\tau\in (0,1)$ in Corollary \ref{charWigner}.
Let us study for simplicity the unweighted case $m=1$. Then by \eqref{C1Rckdistr} for a fixed window $g$
$$\|W_0(f,g)\|_{L^{p,q}}\asymp \|f\|_{L^p}\|\hg\|_{L^q}=\|f\|_{L^p}\|g\|_{\cF L^q}\asymp \|f\|_{L^p}$$
so we are reduced to the $L^p$ norm of the function $f$;
whereas by \eqref{C1Rckdistr2} 
 $$\|W_1(f,g)\|_{L^{p,q}}\asymp \|g\|_{L^p}\|\hf\|_{L^q}=\|g\|_{L^p}\|f\|_{\cF L^q}\asymp \|f\|_{\cF L^q},$$
that is the $\cF L^q$ norm of the function $f$. The above norms are different from the $\mpq$ norm in general, the equality being satisfied  only in the case $p=q=2$:
$$\|f\|_{M^2}=\|f\|_{2}=\|\cF f\|_2.$$
\subsection{Inversion formula for the $\tau$-Wigner distribution}
For $\tau\in [0,1]$ we recall the Moyal's formula for the $\tau$-Wigner distribution \cite[Corollary 1.3.28]{Elena-book}:
	\begin{equation}\label{C1MoyaltauWigner}
\la W_\tau (f_1,g_1), W_\tau (f_2,g_2)\ra =\la f_1,f_2\ra \overline{\la g_1,g_2\ra },\quad f_1,f_2,g_1,g_2\in\lrd.
\end{equation}
\begin{theorem} 
	Assume $\tau\in (0,1)$ and $g_1,g_2\in L^2(\rd)$ with $\la g_1,g_2\ra\not=0$. Then, for any $f\in\lrd$, 
	\begin{equation}\label{invtauW}
	f=\frac{1}{\tau^d}\frac{1}{\la g_2,g_1\ra}\intrdd e^{-\frac{2\pi i}{\tau} x\xi } W_\tau (f,g_1)M_{\frac{\xi}{\tau}}T_{\frac{x}{1-\tau}}Q_\tau g_2\,dxd\xi,
	\end{equation}
	where $Q_\tau$ is defined in \eqref{C1Atau}.
\end{theorem}
\begin{proof} The formula can be inferred from \cite[Corollary 3.17]{CT2020}. For sake of clarity we exhibit a direct proof for the $\tau$-Wigner distribution, following the pattern of \cite[Corollary 3.2.3]{grochenig}.  From the Moyal's formula \eqref{C1MoyaltauWigner} we infer $W_\tau(f,g_1)\in\lrdd$. Moreover, for $g_2\in\lrd$ also the function $Q_\tau g_2\in\lrd$, cf. \eqref{ataulp}; moreover translations and modulations are isometries on $\lrd$ so that $M_{\frac{\xi}{\tau}}T_{\frac{x}{1-\tau}}Q_\tau g_2\in \lrd$ for every $x,\xi \in\rd$. Hence the vector-valued integral
	\begin{equation*}\tilde{f}= \frac{1}{\tau^d}\frac{1}{\la g_2,g_1\ra}\intrdd e^{-\frac{2\pi i}{\tau} x\xi } W_\tau (f,g_1)M_{\frac{\xi}{\tau}}T_{\frac{x}{1-\tau}}Q_\tau g_2\,dxd\xi,
	\end{equation*}
	is a well-defined function $\tilde{f}\in\lrd$ (cf.,e.g., \cite[Section 1.2.4]{Elena-book}). Choose $F\in\lrdd$ and consider the conjugate-linear functional:
	\begin{align*}l(h)&=\frac{1}{\tau^d}\intrdd F\phas  e^{-\frac{2\pi i}{\tau} x\xi }\overline{\la h, M_{\frac{\xi}{\tau}}T_{\frac{x}{1-\tau}}Q_\tau g_2 \ra}\,dxd\xi\\
	&=\intrdd F\phas  e^{-\frac{2\pi i}{\tau} x\xi }\overline{\frac{1}{\tau^d}e^{\frac{2\pi i }{\tau}x\xi}V_{Q_\tau g_2}h(\mathcal{B}_\tau\phas)}\,dxd\xi\\
	&=\intrdd F\phas \overline{W_\tau (h,g_2)\phas}dx\,d\xi,
	\end{align*}
	where we used the connection between STFT and $\tau$-Wigner in \eqref{C1WtauSTFT}. Such functional is  bounded on $\lrd$:
	$$|l(h)|\leq \|F\|_2\|W_\tau (h,g_2)\|_2= \|F\|_2\|h\|_2\|g_2\|_2,$$
	by Moyal's formula \eqref{C1MoyaltauWigner}. So $l(h)\in\lrd$, for every $F\in\lrdd$.  It remains to prove $\tilde{f}=f$. 
	Using Moyal's formula again, 
	\begin{align*}
	\la \tilde{f},h\ra &=\frac{1}{\la g_2,g_1\ra}\intrdd W_\tau(f,g_1)\phas W_\tau(h,g_2)\phas\,dxd\xi \\
	&=\la f,h\ra,
	\end{align*}
	for every $h\in\lrd$, that yields $\tilde{f}=f$ in $\lrd$. 
\end{proof}
\begin{cor}[Inversion formula for the Wigner distribution]
Fix $g_1,g_2\in\lrd$, with $\la g_1,g_2\ra \not=0$. Then for every $f\in\lrd$, 
\begin{equation}\label{invWigner}
f=\frac{2^d}{\la g_2,g_1\ra}\intrdd e^{-4\pi i x\xi}W(f,g_1)\phas M_{2\xi}T_{2x}\cI g_2\,dxd\xi.
\end{equation}
\end{cor}
\begin{proof}
	For $\tau=1/2$ we have $Q_\tau=\cI$, the reflection operator, and this gives the claim.
\end{proof}

In particular, if we consider  the Grossmann-Royer operator $\widehat{T}_{GR}$ defined by
$$\widehat{T}_{GR}\psi(t)=e^{4\pi i \xi (t-x)} \psi(2x-t),$$
for any $\psi\in\lrd$, we can rewrite \eqref{invWigner} as
$$f=\frac{2^d}{\la g_2,g_1\ra}\intrdd e^{-4\pi i x\xi}W(f,g_1)\phas \widehat{T}_{GR} g_2\,dxd\xi,$$
and we recapture the inversion formula for the Wigner distribution proved  in \cite[Prop. 184]{Birkbis}.
\subsection{Modulation spaces and $\tau$-Wigner distributions} Willing to study $W_\tau f=W_\tau(f,f)$ in modulation spaces we are led to consider $W_\tau(f,g)$ with $g\in M^p_v(\rd)$, for $p\geq 1$.
\begin{remark}\label{rem1}
(i) If we consider $f\in M^1_v(\rd)$ in \eqref{chartauwigner}, then among all possible windows we can choose $f=g$. Hence,
$$ f\in M^1_v(\rd)\Leftrightarrow W_\tau (f,f)\in L^1_{v_\tau}(\rdd).$$
(ii) We observe that, as a consequence of Moyal's identity \eqref{C1MoyaltauWigner}, 
$$\|W_\tau (f,g)\|_2=\|f\|_2\|g\|_2,\quad f,g\in\lrd.$$
Since we proved $\|f\|_{M^2}\asymp \|W_\tau(f,g)\|_2$, we infer that for  $M^2(\rd)$ the space of admissible windows can be enlarged from $M^1(\rd)$ to $M^2(\rd)=L^2(\rd)$.
\end{remark}
For $p\in [1,2]$, we denote by $p'$ the conjugate exponent of $p$ ($1/p+1/p'=1$) and set
\begin{equation}\label{expgamma}
\gamma=\frac1p -\frac 1{p'}\in [0,1].
\end{equation}
\begin{theorem}\label{T3-10}
	Consider $1\leq p\leq 2$,  $\tau\in (0,1)$ and  a submultiplicative weight $v$ on $\rdd$ such that  there exists $0<C_1(\tau)\leq C_2(\tau)$  with
	\begin{equation}\label{pesi-eq}
	C_1(\tau) v\phas \leq v(\mathcal{B}_\tau \phas)\leq C_2(\tau)v\phas,\quad \phas\in\rdd.
	\end{equation}
	
	   Fix $g\in M^p_{v^\gamma}(\rd)$. Then
	\begin{equation}\label{Mpvgamma}
	f\in M^p_{v^\gamma}(\rd)\Leftrightarrow W_\tau (f,g) \in L^p_{v^\gamma}(\rdd).
	\end{equation}
\end{theorem}
\begin{proof}
	Fix $g_0\in\cS(\rd)$ such that $\la g_0,g\ra\not=0$ and $\la Q_{\tau} g_0,g\ra\not=0$ (for example take the Gaussian $g_0(t)=e^{-\pi t^2}$). If $g\in M^1_v$ we can use the convolution inequalities in Lemma \ref{convtauwigner} (which still hold for $f, g \in M^1_v(\rd)$ by density argument) and we can write
	\begin{equation}\label{Eq1}
	\|W_\tau (f,g)\|_{L^1_{v_\tau}}\asymp \|W_\tau (f,g)\|_{L^1_{v}}\leq C_\tau \|W_\tau (f,g_0)\|_{L^1_{v}}\|W_\tau(g_0,g)\|_{L^1_{v}}\asymp_\tau \|f\|_{M^1_{v}}\|g\|_{M^1_{v}}.
	\end{equation}
	Now, fix $g\in L^2(\rd)$. By Moyal's formula \eqref{C1MoyaltauWigner} we infer
	\begin{equation}\label{Eq2}
	\|W_\tau (f,g)\|_{2}=\|f\|_{M^2} \|g\|_{M^2}.
	\end{equation}
	for every  $f\in L^2(\rd)$. Observe that the inclusion relations for modulation spaces (cf. \cite[Theorem 2.4.17]{Elena-book}) give $g\in M^1_v(\rd)\hookrightarrow L^2(\rd)$,
	so that the  mapping $W_\tau  $ (linear with respect to the first component and anti-linear  to the second one)
	is bounded from $M^1_{v}(\rd)\times M^1_{v}(\rd) $ into $L^1_{v}(\rdd)$ and from $M^2(\rd)\times M^2(\rd)$ into $L^2(\rdd)$. By complex interpolation of modulation spaces (cf. \cite[Proposition 2.3.16]{Elena-book}) and Lebesgue spaces (cf. \cite{triebel} ) we infer, for $\theta\in [0,1]$,
	$$[M^1_v, M^2]_\theta=M^p_{v^{1-\theta}},\quad [L^1_{v},L^2]_\theta=L^p_{v^{1-\theta}},$$
	where
	$$\frac1p=1-\frac{\theta}{2}\quad\Leftrightarrow \frac{\theta}{2}=\frac{1}{p'}$$
	and $$ 1-\theta=\frac1p-\frac{\theta}{2}=\frac1p-\frac{1}{p'}=\gamma.$$ (observe $p\in [1,2]$) we infer that the linear mapping
	$$W_\tau: M^p_{v^\gamma}(\rd)\times M^p_{v^\gamma}(\rd)\to L^p_{v^\gamma}(\rdd)$$
	is well defined and bounded. Vice versa, using  the convolution inequalities in Lemma \ref{convtauwigner} we obtain
	\begin{align*}\|f\|_{M^p_{v^\gamma}}&\asymp_\tau \|W_\tau(f,g_0)\|_{L^p_{v^\gamma}}\leq \frac{1}{(1-\tau)^d}\frac{1}{|\la Q_\tau g_0, g\ra| }\|W_\tau(f,g)\|_{L^p_{v^\gamma}}\|W_\tau g_0\|_{L^1_{v^\gamma}}\\
	&\leq C(\tau,g,g_0)\|W_\tau(f,g)\|_{L^p_{v^\gamma}}.
	\end{align*}
	Hence we obtain the claim, since, for $g_0\in\cS(\rd)$, the $\tau$-Wigner $W_\tau g_0$ is in $\cS(\rdd)\subset L^1_{v^\gamma}(\rdd)$ (cf. \cite[Corollary 1.3.26(i)]{Elena-book}).
\end{proof}

Condition \eqref{pesi-eq} is satisfied by the submultiplicative weights $v_s$ defined in \eqref{weightvs}. In fact, defining $$\tau_{min}=\min\{1-\tau,\tau\}\in(0,1),\quad \tau_{max}=\max \{1-\tau,\tau\}\in (0,1)  $$
\begin{align*}
\frac{1}{\tau^d_{max}}|\phas|\leq |\mathcal{B}_\tau\phas|\leq \frac{1}{\tau^d_{min}} |\phas|,
\end{align*}
so that
\begin{equation}\label{taueqpesi}
v_s\asymp_\tau v_s(\mathcal{B}_\tau\cdot).
\end{equation}

Observe that if we fix the window $g\in M^1_v(\rd)$ we have

\begin{theorem}\label{T3-10new}
	Consider $1\leq p\leq \infty$,  $\tau\in (0,1)$ and  a submultiplicative weight $v$ on $\rdd$ satisfying \eqref{pesi-eq}.   Fix $g\in M^1_{v}(\rd)$. Then
	\begin{equation}\label{Mpvgammanew}
	f\in M^p_{v}(\rd)\Leftrightarrow W_\tau (f,g) \in L^p_{v}(\rdd).
	\end{equation}
\end{theorem}
\begin{proof}
	It immediately follows from Corollary \ref{charWigner}.
\end{proof}

The following theorems extend the result for the Wigner distribution \cite[Theorem 4.4.1]{Elena-book} (cf. \cite[Theorem 4]{Elena-Note2021}) to any $\tau\in (0,1)$. They estimate the modulation norm of $W_\tau(f,g)$ in terms of the modulation norms of $f,g$ and play a crucial role in the final Section $5$. 
 \begin{theorem}\label{C4wigestp}
	Assume $\tau\in [0,1]$ and indices  $p_1,q_1,p_2,q_2,p,q\in [1,\infty]$ such that
	\begin{equation}\label{C4WIR}
	p_i,q_i\leq q,  \ \quad i=1,2
	\end{equation}
	and 
	\begin{equation}\label{C4Wigindexsharp}
	\frac1{p_1}+\frac1{p_2}\geq \frac1{p}+\frac1{q},\quad \frac1{q_1}+\frac1{q_2} \geq \frac1{p}+\frac1{q}.
	\end{equation}
Consider $s\in\bR$, the weight functions 
	$v_s,   1\otimes v_s$ defined in \eqref{weightvs}, \eqref{weight1tensorvs}, respectively. If
	 $f\in M^{p_1,q_1}_{v_{|s|}}(\Ren)$ and
	$g\in M^{p_2,q_2}_{v_s}(\Ren)$, then  
 $W_\tau(f,g)\in	M^{p,q}_{1\otimes v_s}(\Renn)$, and the following estimate holds \begin{equation}\label{C4wigest}
	\| W_\tau(f,g)\|_{M^{p,q}_{1\otimes v_s}}\lesssim
	\|f\|_{M^{p_1,q_1}_{v_{|s|}}}\| g\|_{M^{p_2,q_2}_{v_s}}.
	\end{equation}
\end{theorem}
\begin{proof}
	The proof uses the  formula of the STFT of the $\tau$-Wigner distribution  recalled in Lemma \ref{lem:STFT of tauWig}. Namely, 
	\begin{align*}
\left\Vert W_{\tau}(f,g)\right\Vert _{M_{1\otimes v_s}^{p,q}}&=\left(\int_{\mathbb{R}^{2d}}\left(\int_{\mathbb{R}^{2d}}\left|{V}_{\Phi_{\tau}}W_\tau(f,g)\left(z,\zeta\right)\right|^p d z\right)^{\frac{q}{p}}v_s\left(\zeta\right)^{q}d \zeta\right)^{\frac{1}{q}}\\
	&=\left(\int_{\mathbb{R}^{2d}}\left(\int_{\mathbb{R}^{2d}}\left|V_{\varphi_{1}}f\left(z-\mathcal{T}_{1-\tau}\zeta\right)\right|^p
	\cdot\left|V_{\varphi_{2}}g\left(z+\mathcal{T}_\tau\zeta\right)\right|^p d z\right)^{\frac{q}{p}}v_s\left(\zeta\right)^{q}d \zeta\right)^{\frac{1}{q}}.
	\end{align*}
	The substitution $z'=z+\mathcal{T}_\tau\zeta$, the properties
	of $\mathcal{T}_\tau$ provided in Lemma \ref{atau}
	(in particular $\mathcal{T}_\tau+\mathcal{T}_{1-\tau}=J$) yield
	\begin{align*}
\left\Vert W_{\tau}(f,g)\right\Vert _{M_{1\otimes v_s}^{p,q}}&
	=\left(\int_{\mathbb{R}^{2d}}\left(\int_{\mathbb{R}^{2d}}\left|V_{\varphi_{1}}f\left(z'-J\zeta\right)\right|^p\cdot\left|V_{\varphi_{2}}g\left(z'\right)\right|
	^pd z'\right)^{\frac{q}{p}}v_s\left(\zeta\right)^{q}d \zeta\right)^{\frac{1}{q}}\\
	&=\left(\int_{\mathbb{R}^{2d}}\left[\left|V_{\varphi_{2}}g\right|^p*\left|(V_{\varphi_{1}}f\right)^{*}|^{p}\left(J\zeta\right)\right]^{\frac{q}{p}}v_s\left(\zeta\right)^{q
	}d \zeta\right)^{\frac{1}{p}}\\
	&	=\left\Vert \left|V_{\varphi_{2}}g\right|*\left|(V_{\varphi_{1}}f)^{*}\right|^p\right\Vert^{\frac1p} _{L_{v_{ps}}^{\frac{q}{p}}},
	\end{align*}
	since $v_s(J\zeta)=v_s(\zeta)$.   The rest of the proof follows the pattern of the corresponding one for the Wigner distribution in \cite[Theorem 4.4.2]{Elena-book}.
\end{proof}
\begin{remark}\label{rem3.13}
	In this framework,  the recent contribution by Guo et al.	\cite[Theorem 1.1]{guo2020characterization} shows boundedness results for $\tau$-Wigner distributions on modulation spaces, where they consider different  weights for the functions $f,g$ of the type $v_{t,s}(z_1,z_2)=\la z_1\ra^{t}\la z_2\ra^s$.
\end{remark}

\begin{theorem}\label{C4wigestptempo}
	Assume $\tau\in [0,1]$, $1\leq p\leq \infty$ and $s\geq 0$ and  the weight functions 
	$v_s,   v_s\otimes 1$ defined in \eqref{weightvs}, \eqref{weight1tensorvs}, respectively.
	 If  $f,g\in M^{p}_{v_{s}}(\Ren)$, then  
	$W_\tau(f,g)\in	M^{p}_{v_s\otimes 1}(\Renn)$, with
\begin{equation}\label{C4wigestppp}
\| W_\tau(f,g)\|_{M^{p}_{v_s\otimes 1}}\lesssim
\|f\|_{M^{p}_{v_{s}}}\| g\|_{M^{p}_{v_s}}
\end{equation}	
\end{theorem}	
\begin{proof}
Consider  $f,g\in  M^{p}_{v_{s}}(\Ren)$. In view of Theorem \cite[2.3.27]{Elena-book} 
we have
\begin{align*}
\| W_\tau(f,g)\|_{M^{p}_{v_s\otimes 1}}&=\| \widehat{W_\tau(f,g)}\|_{M^{p}_{1\otimes v_s}}.
\end{align*}
Using \cite[Proposition 1.3.27]{Elena-book}, for $\tau\in [0,1]$, we can write $$\widehat{W_\tau(f,g)}=\cF[W(f,g)\ast \sigma_\tau],$$
where the kernel $\sigma_\tau\in\cS'(\rdd)$ is given by \eqref{kernel of tau-Wigner distribution}.
Observe that the convolution $W(f,g)\ast \sigma_\tau$ is well defined for $f,g\in  M^{p}_{v_{s}}(\Ren)$ since  $W(f,g)\in M^{p}_{1\otimes v_{s}} (\rdd)$ by Theorem \ref{C4wigestp},  $\sigma_\tau\in M^{1,\infty}(\rdd)$ (see \cite[Prop. 4.1]{BJpseudo}), and the convolution relations for modulation spaces give $M^p_{1\otimes v_{s}}(\rdd) \ast  M^{1,\infty}(\rdd)\hookrightarrow M^{p}(\rdd)$. 
Let us first study the case  $\tau=1/2$, where $$\widehat{W_\tau(f,g)}\phas=\cF W(f,g)\phas.$$

Now $$\cF W(f,g)\phas=\cF_\sigma W(f,g)(-J\phas)=A (f,g)(-J\phas)$$
(where in the last equality we used \cite[Lemma 1.3.11]{Elena-book} for the ambiguity function $A(f,g)$)
and
$$ A (f,g)(-J\phas)=2^{-d}W(f,\cI g)(-\frac{J}{2}\phas).$$ 
Using the easy-verified formula for the STFT of a dilated function $f_D(t):=f(Dt)$, $D$ an invertible $d\times d$ matrix,
\begin{equation}\label{STFTdil}
V_\f f_D \phas =|\det D|^{-1} V_{\f_{D^{-1}}} f(Dx,(D^*)^{-1}\xi).
\end{equation}
and taking in our context $\Phi(z)=e^{-2\pi z^2}$ so that for $D=-\frac{J}{2}$, we get $\Phi_{D^{-1}}(z)=e^{-\frac12\pi z^2}$. 
Hence \begin{align*}\| \widehat{W_\tau(f,g)}\|_{M^{p}_{1\otimes v_s}}&=2^{-d}\|W(f,\cI g)(-\frac{J}{2}\cdot)\|_{M^{p}_{1\otimes v_s}}\\
&\asymp \left(\int_{\bR^{4d}}| V_{\Phi_{D^{-1}}} W(f,\cI g)(-\frac{J}{2}z, -2J\zeta )|^p v_s^p(\zeta)dz\,d\zeta\right)^{\frac1p} \\
&\asymp \left(\int_{\bR^{4d}}| V_{\Phi_{D^{-1}}} W(f,\cI g)(z, \zeta )|^p v_s^p(\zeta)dz\,d\zeta\right)^{\frac1p}\\
&=\|W(f,\cI g)\|_{M^p_{1\otimes v_s}} \end{align*}
since $v_s(-2J\zeta)=v_s(2\zeta)\asymp v_s(\zeta)$.
Finally, the conclusion follows from Theorem \ref{C4wigestp}, observing that, by relation \eqref{STFTdil},
$$\|\cI g\|_{M^p_{v_s}}=\| g\|_{M^p_{v_s}}.$$

Case $\tau\not=1/2$. Here we can write
$$\widehat{W_\tau(f,g)}\phas=[\cF W(f,g)\cdot\cF \sigma_\tau]\phas$$
with 
$$\cF \sigma_\tau\phas= e^{-\pi i (2\tau -1)x\xi},$$
(see \cite[Formula (1.114)]{Elena-book}).
By \cite[Proposition 4.1]{BJpseudo} we infer $\cF \sigma_\tau\in W(\cF L^1,L^\infty)(\rdd)$ and using $M^{p}_{1\otimes v_s}=W(\cF L^p, L^p_{v_s})$
and the pointwise product for  Wiener amalgam spaces
$$ W(\cF L^1, L^\infty) \cdot W(\cF L^p , L^p_{v_s})\subset W(\cF L^p,  L^p_{v_s})$$ 
we get 
\begin{align*}
\|\cF W(f,g)\cdot\cF \sigma_\tau\|_{M^{p}_{1\otimes v_s}}&\lesssim \|\cF W(f,g)\|_{M^{p}_{1\otimes v_s}}\|\cF \sigma_\tau\|_{W(\cF L^1,L^\infty)}\\
&\asymp \|\cF W(f,g)\|_{M^{p}_{1\otimes v_s}}\|\sigma_\tau\|_{M^{1,\infty}},
\end{align*}
and the result follows from the case $\tau=1/2.$ 
\end{proof}	
 
\subsection{Modulation spaces and $\tau$-Wigner distributions (conclusions)}
Let us summarise the results of this section. Observe that for the weight $v_s$ we infer
$$v_s^\gamma(z)=(1+|z|^2)^{\frac{s\gamma}{2}}=v_{s\gamma}(z),\quad z\in\rdd,$$
so that we are reduced to the same type of polynomial weight $v_{s'}$, with $$0\leq s'=s\gamma\leq s.$$
	
\begin{corollary}\label{C3.14}
Consider $\tau\in (0,1)$,  $1\leq p\leq 2$, $\gamma$ defined in \eqref{expgamma}, $s\geq0$ and the weights $v_s, v_s\otimes 1$ in \eqref{weightvs}, \eqref{weight1tensorvs}, respectively. Fix $g\in M^p_{v_{s\gamma}}(\rd)$. Then the following conditions are equivalent:
\begin{itemize}
	\item [\textit{(i)}] $f \in M^p_{v_{s\gamma}}(\rd)$
	\item [\textit{(ii)}] $W_\tau (f,g) \in L^p_{v_{s \gamma}}(\rdd)$
	\item [\textit{(iii)}] $W_\tau (f,g) \in M^p_{v_{s \gamma}\otimes 1}(\rdd)$,
\end{itemize}
where the exponent $\gamma$ is defined in \eqref{expgamma}.
\end{corollary}
\begin{proof}
	$ (i)\Leftrightarrow (ii).$ It immediately follows from Theorem \ref{T3-10}  and the weight equivalence in \eqref{taueqpesi}. \par
	$ (i) \Rightarrow (iii).$   It is proved in Theorem \ref{C4wigestptempo}.
		$ (iii) \Rightarrow (ii).$ It follows by the inclusion relations $M^p_{v_{s \gamma}\otimes 1}(\rdd)\subset L^p_{v_{s \gamma}}(\rdd)$, for $1\leq p\leq 2$, cf. \cite[Proposition 2.9]{Toftweight2004}.
\end{proof}
Since the previous result holds true for every $s\geq 0$, if we avoid the case $\gamma=0$ that corresponds to $p=2$ we can  state the simpler  characterization: 
 \begin{corollary}\label{cor-corC3.14}
 Consider $\tau\in (0,1)$,  $1\leq p< 2$, $s\geq0$ and the weights $v_s, v_s\otimes 1$ in \eqref{weightvs}, \eqref{weight1tensorvs}, respectively.  Fix $g\in M^p_{v_{s}}(\rd)$. Then the following conditions are equivalent:
 \begin{itemize}
 	\item [\textit{(i)}] $f \in M^p_{v_{s}}(\rd)$
 	\item [\textit{(ii)}] $W_\tau (f,g) \in L^p_{v_{s}}(\rdd)$
 	\item [\textit{(iii)}] $W_\tau (f,g) \in M^p_{v_{s }\otimes 1}(\rdd)$.
 \end{itemize}
 \end{corollary}
 
 Note that the result above holds true for any \emph{fixed} window $g\in M^p_{v_s}(\rd)$. 
 One could be tempted to put $f=g$ in the characterization above, and  be misled by thinking that 
 $$\|f\|_{M^p_{v_{s}}}\asymp \|W_\tau f\|_{L^p_{v_{s }}}\quad\mbox{or}\quad \|f\|_{M^p_{v_{s}}}\asymp \|\sqrt{|W_\tau f|}\|_{L^p_{v_{s }}} .$$
 This is not even the case when $s=0$. As an example,  consider $g(t)=\varphi (t)=e^{-\pi t^{2}}$ and its rescaled version $f(t)=\varphi _{\sqrt{\lambda} }(t)=e^{-\pi \lambda t^2}$, and  $\tau=1/2$.
 Then (see, e.g. \cite[Lemma 1.3.34]{Elena-book})
 \begin{equation}
 W(\varphi _{\sqrt{\lambda}},\varphi )(x,\xi )=\frac{2^{d}}{(\lambda +1)^{\frac{d%
 		}{2}}}e^{-\frac{4\pi \lambda }{\lambda +1}x^{2}}e^{-\frac{4\pi }{\lambda +1%
 	}\xi ^{2}}e^{4\pi i\frac{\lambda -1}{\lambda +1}x\xi }.
 \label{wign1l}
 \end{equation}
 By Corollary \ref{charWigner} 
 \begin{equation}\label{C3-e1}
 \|\varphi _{\sqrt{\lambda}}\|_{M^p}\asymp \|W(\varphi _{\sqrt{\lambda}},\varphi)\|_{L^p}
\asymp \|W(\varphi ,\varphi _{\sqrt{\lambda}})\|_{L^p} \asymp   \frac{(\lambda+1)^{\frac{d}{p}-\frac{d}{2}}}{\lambda^{\frac{d}{2p}}}
 \end{equation}
 whereas an easy computation shows (see \cite[(4.20)]{grochenig})
 $$W(\varphi _{\sqrt{\lambda}},\varphi _{\sqrt{\lambda}})=2^{\frac d2}\lambda^{-\frac d 4} e^{-2\pi\lambda x^2}e^{-\frac{2\pi }{\lambda} \xi^2}$$
  \begin{equation}\label{C3-e2}
 \|W(\varphi _{\sqrt{\lambda}},\varphi _{\sqrt{\lambda}})\|_{L^p}\asymp \lambda^{-\frac d2}. 
 \end{equation}
 
 Hence it is clear that the norms in \eqref{C3-e1} and \eqref{C3-e2} behave differently as the parameter $\lambda $ goes to $0$ or to $+\infty$. In particular, the norm in \eqref{C3-e2} does not even depend on the exponent $p$.
 


 \begin{corollary}\label{C3.14bis}
	Assume $\tau\in [0,1]$, $1\leq p\leq \infty$ and $s\geq 0$ and  the weight functions 
$v_s$ defined in \eqref{weightvs}.
If    $f,g \in M^p_{v_s}(\rd)$ then 
  $ W_\tau(f,g)\in M^p_{v_s}(\rdd)$ with
  \begin{equation}\label{C3pstimaglobale}
  \|W_\tau(f,g)\|_{M^p_{v_s}}\lesssim \|f\|_{M^p_{v_s}} \|g\|_{M^p_{v_s}}.
  \end{equation}
 \end{corollary}	
 \begin{proof}
 If  $f,g \in M^p_{v_s}(\rd)$	from Theorems \ref{C4wigestp} and \ref{C4wigestptempo} we infer that $W_\tau f, W_\tau (f,g) \in M^p_{v_s\otimes 1}(\rdd)\cap M^p_{1\otimes v_s}(\rdd)$.
 For $s\geq 0$ we have the equivalence
 $$v_s(z_1,z_2)\asymp v_s(z_1)+v_s(z_2),\quad \forall z_1,z_2\in \rdd,$$
 hence, for $1\leq p<\infty$, $v_s(z_1,z_2)^p\asymp v_s(z_1)^p+v_s(z_2)^p$, for every $z_1,z_2\in\rdd$. For every fixed $\Phi\in\cS(\rdd)$, we can write, for $p<\infty$,
 \begin{align*}
 \| W_\tau &(f,g) \|^p_{M^p_{v_s}}\asymp \int_{\bR^{4d}}|V_\Phi W_\tau (f,g)(z_1,z_z)|^pv_s(z_1,z_2)^p dz_1dz_2\\
 &\lesssim \int_{\bR^{4d}}|V_\Phi W_\tau (f,g)(z_1,z_z)|^pv_s(z_1)^p dz_1dz_2+\int_{\bR^{4d}}|V_\Phi W_\tau (f,g)(z_1,z_z)|^pv_s(z_2)^p dz_1dz_2\\
 &\asymp \|W_\tau (f,g)\|^p_{M^p_{v_s\otimes 1}}+ \|W_\tau (f,g)\|^p_{M^p_{1\otimes v_s}}\\
 &\lesssim \|f\|_{M^p_{v_s}}^p\|g\|_{M^p_{v_s}}^p.
 \end{align*}
 The case $p=\infty$ is similar. This concludes the proof.
 \end{proof}
 
\section{Metaplectic Operators and $\cA$-Wigner representations}
In this section we highlight a new viewpoint for time-frequency representations: they can be defined as images of metaplectic operators. This approach might shed more light  on  the roots of  Time-frequency Analysis and Quantum Harmonic Analysis.
\begin{definition}\label{def4.1}
Consider a $4d\times 4d$ symplectic matrix $\cA\in Sp(2d,\bR)$ and define the time-frequency representation \textbf{$\cA$-Wigner} of $f,g\in\lrd$  by 
\begin{equation}\label{WignerA}
W_\cA (f,g)=\mu(\cA) (f\otimes \bar{g}),\quad f,g\in\lrd.
\end{equation}
Observe that for $f,g\in\lrd$, the tensor product $f\otimes \bar{g}$ acts continuously from $L^2(\rd)\times L^2(\rd)$ to $\lrdd$ and $\mu(\cA)$ is a unitary operator on $\lrdd$, hence $W_{\cA}$ is a well-defined mapping from $L^2(\rd)\times L^2(\rd)$ into $L^2(\rdd)$. We set $W_\cA f:= W_\cA (f,f)$.
\end{definition}
Note that the metaplectic operator $\mu(\cA)$ is defined up to a multiplicative factor, and $W_\cA$ in \eqref{WignerA} depends on its choice. By abuse, in the notation we omit to specify the choice and limit to the dependence on $\cA$. When appropriate we shall detail the phase factor, see in particular the following class of examples.\par

We are interested in matrices $\cA \in Sp(2d,\bR)$ such that 
\begin{equation}\label{metapf2dil}
\mu(\cA)=\cF_2 \mathfrak{T}_{L}
\end{equation}
where $\cF_2$ is the partial Fourier transform with respect to the second variables $y$ defined in \eqref{FT2}
and the change of coordinates $\mathfrak{T}_{L}$ is defined in \eqref{Ltransf}.  
 If we introduce  the symplectic matrix
 \begin{equation}\label{A-f2}
 \mathcal{A}_{FT2}=\left(\begin{array}{cc}
 A_{11} & A_{12}\\
A_{21} & A_{22}
 \end{array}\right)\in Sp(2d,\bR),
 \end{equation}
 where $A_{11},A_{12},A_{21},A_{22}$ are the $d\times d$ matrices 
  \begin{equation}\label{A-f2bis}
{A}_{11}=A_{22}=\left(\begin{array}{cc}
 I_{d\times d} &0_{d\times d}\\
 0_{d\times d} &0_{d\times d}
 \end{array}\right), \quad {A}_{12}=\left(\begin{array}{cc}
 0_{d\times d} &0_{d\times d}\\
 0_{d\times d} &I_{d\times d}
 \end{array}\right), \quad {A}_{21}=-{A}_{12}
 \end{equation}
 then it was  shown by Morsche and Oonincx \cite[Sec. 6.2]{MO2002} that, for a choice of the phase factor, 
 \begin{equation}\label{MOf2}
 \mu(\cA_{FT2})=\cF_2.
 \end{equation}
 Moreover, for $L\in GL(2d,\bR)$, we have (see, e.g., \cite[Prop. 1.1.3]{Elena-book})
 \begin{equation}\label{MotL}
\cD_L= \left(\begin{array}{cc}
L^{-1} &0_{d\times d}\\
0_{d\times d} & L^T
\end{array}\right)\in Sp(2d,\bR)
 \end{equation}
 and with a choice of the phase factor
 \begin{equation}\label{AdL}
\mu(\cD_L)F(x,y)=\sqrt{|\det L|}F(L(x,y))=\mathfrak{T}_{L} F(x,y),\quad F\in\lrdd.
 \end{equation}
 An easy computation shows
\begin{equation}\label{f2D}
\c{\bf A}:= \mathcal{A}_{FT2}\cD_L\in Sp(2d,\bR),
\end{equation}
 where 
 
$${\bf A}=  \left(\begin{array}{cc}
A_{11}L^{-1} &A_{12}L^T\\
A_{21} L^{-1} & A_{11} L^T
\end{array}\right), \quad {\bf A}^{-1}=  \left(\begin{array}{cc}
L\, A_{11} &L \,A_{21}\\
L^{-T}A_{12}  & L^{-T} A_{11} 
\end{array}\right).
 $$
 Hence the symplectic matrix $\cA={\bf A}$  satisfies the equality in \eqref{metapf2dil}.
 
 \begin{remark}\label{rem1E}
(i) Consider the linear operator  $\mathfrak{T}_a$  defined by
 \begin{equation}\label{Changecoord}
 \mathfrak{T}_a F(x,y)=F(y,y-x)\quad x,y\in\rd.
 \end{equation}
 Observe that $\mathfrak{T}_a=\mathfrak{T}_L$, with 
 $$L=\left(\begin{array}{cc}
 0_{d\times d} &I_{d\times d}\\
 -I_{d\times d} &I_{d\times d}
 \end{array}\right).$$
We can then regard the STFT as $\cA$-Wigner representation according to \eqref{WignerA} in Definition \ref{def4.1}. Namely, for $f,g\in\lrd$, 
 \begin{equation}\label{SFTFsesq}
 V_g f=\cF_2 \mathfrak{T}_a(f\otimes \bar{g})=\mu({\bf A_{ST}})(f\otimes \bar{g}),
 \end{equation}
 where ${\bf A_{ST}}:=\mathcal{A}_{FT2}\cD_L$ is explicitly computed as
 \begin{equation}\label{A-FT} {\bf A_{ST}}=\left(\begin{array}{cccc}
 I_{d\times d} &- I_{d\times d}&0_{d\times d}&0_{d\times d}\\
 0_{d\times d}&0_{d\times d}& I_{d\times d} &I_{d\times d}\\
 0_{d\times d}&0_{d\times d}& 0_{d\times d} &-I_{d\times d}\\
 -I_{d\times d}&0_{d\times d}& 0_{d\times d} &0_{d\times d}\\
 \end{array}\right).
 \end{equation}
 (ii) For $\tau\in [0,1]$,  $L=L_\tau$ in \eqref{A-taumatr}, the symplectic matrix ${\bf A}_{\tau} :=\cA_{FT2}\cD_L$ in \eqref{f2D} can be explicitly computed as 
\begin{equation}\label{Aw-tau} {\bf A}_{\tau}=\left(\begin{array}{cccc}
 	(1-\tau)I_{d\times d} &\tau I_{d\times d}&0_{d\times d}&0_{d\times d}\\
 	0_{d\times d}&0_{d\times d}&\tau I_{d\times d} &-(1-\tau)I_{d\times d}\\
 	0_{d\times d}&0_{d\times d}& I_{d\times d} &I_{d\times d}\\
 	-I_{d\times d}&I_{d\times d}& 0_{d\times d} &0_{d\times d}\\
 \end{array}\right).
 \end{equation}
 This allows the representation of the $\tau$-Wigner as $\cA$-Wigner  distribution with matrix $\cA={\bf A}_{\tau}$:
 \begin{equation}\label{mutau-wigner}
W_\tau(f,g)\phas=\mu({\bf A}_{\tau})(f\otimes \bg)\phas,\quad f,g\in\lrd,\quad\phas\in\rdd.
 \end{equation}
 \end{remark}
  In particular $W(f,g)=\mu({\bf A}_{1/2})(f\otimes \bg)$.\par
 \textbf{$\cA$-pseudodifferetial operators.} Schwartz' kernel theorem  states, in the framework of tempered distributions, that every linear continuous operator $T:\cS(\rd)\to\cS'(\rd)$ can be regarded as an integral operator in a generalized sense, namely 
  \[
  \langle Tf,g\rangle=\langle K,g\otimes \overline{f}\rangle,\qquad f,g\in\cS(\rd),
  \]
  in the context  of distributions, for some kernel $K\in\cS'(\rdd)$, and vice versa \cite{HormanderI}.
  For $\cA \in Sp(2d,\bR)$, using $\mu(\cA)\mu(\cA)^{-1}=I$, the identity operator, we can write
\[
\langle Tf,g\rangle=\langle \mu(\cA)^{-1}\mu (\cA)K,g\otimes \overline{f}\rangle=\langle \mu(\cA)K,\mu (\cA)(g\otimes \overline{f})\rangle=\langle \mu(\cA)K,W_\cA(g\otimes \overline{f})\rangle
\]
 The equalities above  provide a new definition of a pseudodifferential operator with symbol $\sigma_\cA$ related to the symplectic matrix $\cA$.
 
 We define a  \textbf{$\cA$-pseudodifferential operator} related to the  $\cA$-Wigner representation the mapping $Op(\sigma_\cA):\cS(\rd)\to\cS'(\rd)$ given by
 \begin{equation}\label{capseudo}
\la Op(\sigma_\cA)f,g\ra=\la \sigma_\cA, W_\cA(g\otimes \overline{f})\rangle, 
 \end{equation}
  where the $\cA$-Wigner is defined in \eqref{WignerA} and the  symbol $\sigma_\cA$ is given by
  \begin{equation}\label{sigmaA}
 \sigma_\cA= \mu(\cA)K,
  \end{equation}
  with $K$ integral  kernel of the operator.

   In this paper we limit ourselves to the case $\cA={\bf A}_{\tau}$. In the sequel we shall also write for short
   \begin{equation}\label{tildeW}
   	\widetilde{W}_\tau (F):=\mu({\bf A}_{\tau})F,\quad F\in L^2(\rdd).
   \end{equation}
For $F=f\otimes\bg$, $f,g\in\lrd$ we come back to the cross-$\tau$-Wigner distribution 	$\widetilde{W}_\tau (f\otimes\bg)=W_\tau(f,g)$. Note that the definition of $\tau$-operators in \eqref{C21} of the Introduction is now a particular case of \eqref{capseudo}, in view of \eqref{mutau-wigner}. In detail we have
\begin{equation}\label{optau}
Op_\tau(\sigma)f(x):=Op(\sigma_{{\bf A}_{\tau}})f(x)=\intrdd e^{2\pi i(x-y)\xi}\sigma((1-\tau)x+\tau y,\xi)f(y)dyd\xi.
\end{equation}
\begin{lemma}\label{opbtau}
For $\tau\in [0,1]$, $a\in \cS'(\rdd)$, $f,g\in\cS(\rd)$, we have
\begin{equation}\label{eq}
(Op_\tau(a)f)\otimes \bg=Op_\tau(\sigma)(f\otimes\bg),
\end{equation}
with
\begin{equation}\label{sigma-tensore}
\sigma(r,y,\rho,\eta)=a(r,\rho)\otimes 1_{(y,\eta)},\quad r,\rho,y,\eta\in\rd,
\end{equation}
and $1_{(y,\eta)}\equiv 1$, for every $(y,\eta)\in\rdd$. Besides, the result is still valid if we replace $\cS',\cS$ by the modulation spaces $M^{\infty}$, $M^{1}$.
\end{lemma}
\begin{proof} The operators are well defined by the Schwartz' kernel theorem and the equality in \eqref{eq} is a straightforward computation. The case of modulation spaces is analogous, one has to use the kernel theorem for modulation spaces, cf. \cite[Sec. 3.3]{Elena-book}. 
	\end{proof}

In what follows we need the inverse matrix ${\bf A}_{\tau}^{-1}$ of ${\bf A}_{\tau}$, that can be easily computed as 
\begin{equation}\label{Aw-tau-inverse} {\bf A}_{\tau}^{-1}=\left(\begin{array}{cccc}
I_{d\times d} & \,0_{d\times d}&0_{d\times d}&-\tau\, I_{d\times d}\\
I_{d\times d}&0_{d\times d}& 0_{d\times d} &(1-\tau)I_{d\times d}\\
0_{d\times d}& I_{d\times d}& (1-\tau)I_{d\times d} &0_{d\times d}\\
0_{d\times d}&-I_{d\times d}& \tau \, I_{d\times d} &0_{d\times d}\\
\end{array}\right).
\end{equation}
\begin{lemma}\label{propNtau}
	For $\tau\in [0,1]$ consider the matrix 
	\begin{equation}\label{ntau}
	N_\tau=\left(\begin{array}{cc}
	I_{2d\times 2d} &0_{2d\times 2d}\\
	C_{\tau}&I_{2d\times2 d}\\
	\end{array}\right) \quad \mbox{where}\quad C_\tau=(\tau-1/2)\left(\begin{array}{cc}
	0_{d\times d} &I_{d\times d}\\
I_{d\times d}&0_{d\times d}\\
	\end{array}\right) .
	\end{equation}
	Observe that $C_\tau^T=C_\tau$ and $N_\tau\in Sp(2d,\bR)$. Moreover,\\
	(i) $N_\tau^{-1}=N_{1-\tau}$\\
	(ii) For $f\in\lrdd$, $\mu(N_\tau)f=e^{-2\pi i (\tau-1/2)\Phi}$f , with $\Phi(x,\xi)=x\xi$, $x,\xi\in\rd$.\\
	(iii) For $f\in\lrdd$, $\mu(N_\tau^{-T})f=\cF^{-1}e^{-2\pi i (\tau-1/2)\Phi}\cF f$.
\end{lemma}
\begin{proof}
	Item $(i)$ is a simple computation. Item $(ii)$ follows by formula \eqref{lower}.\\
	Let us prove Item $(iii)$.  From the definition of a  symplectic matrix we get $$N^{-T}_\tau=J^{-1}N_\tau J.$$
	Applying the metaplectic representation and using \eqref{muJ} we obtain
$$\mu(N^{-T}_\tau)=\cF^{-1}\mu(N_\tau)\cF,$$
that gives the claim.
\end{proof}
\begin{theorem}\label{Tmain}
Consider $a\in\cS'(\rd)$. Then for every $f,g\in \cS(\rd)$, for $\tau_1,\tau_2\in [0,1]$, we have
\begin{equation}\label{casonoWeyl}
W_{\tau_1}(Op_{\tau_2}(a)f,g)=Op_{\tau_2}(b)W_{\tau_1}(f,g),
\end{equation}
where
\begin{equation}\label{simbolob}
b=\mu(N_{\tau_2}^T)[(\mu(N_{\tau_2}^{-T})\sigma)\circ {\bf A}^{-1}_{\tau_1}],
\end{equation}
where $N_{\tau_2}\in Sp(2d,\bR)$ is defined in \eqref{ntau} and $\sigma$ in \eqref{sigma-tensore}.
In particular, for $\tau_2=1/2$, we write $Op_{w}:=Op_{1/2}$ and the equality in \eqref{casonoWeyl} becomes
\begin{equation}\label{casoWeyl}
W_{\tau_1}(Op_w(a)f,g)=Op_{w}(b)W_{\tau_1}(f,g),
\end{equation}
where
\begin{equation}\label{simboloWeyl}
b(x,\xi,u, v)=\sigma( {\bf A}^{-1}_{\tau_1}(x,\xi,u,v))= a(x-{\tau_1}v,\xi+(1-{\tau_1}) u),\quad x,\xi,u,v\in\rd.
\end{equation}
\noindent
\end{theorem}
To be definite about the notation for variables in Theorem \ref{Tmain} and subsequent proof: the linear map ${\bf A}_{\tau_1}$ acts from  $(r,y)$, with respective dual variables $(\rho,\eta)$, to $\phas$ with dual variables $(u,v)$. By standard Weyl quantization, the right-hand side of \eqref{casoWeyl} reads 
$$Op_w(b)W_{\tau_1}(f,g)\phas=a(x-\frac{1}{2\pi}\tau_1 D_\xi,\xi+\frac{1}{2\pi}(1-\tau_1)D_x)W_{\tau_1}(f,g).$$
As a particular case, we obtain the $\tau$-Moyal operators, $\tau\in [0,1]$,
\begin{equation*}
W_\tau(x_j f,g)=(x_j-\frac1{2\pi}\tau D_{\xi_j})W_\tau(f,g),
\end{equation*}
\begin{equation*}
W_\tau(D_{x_j} f,g)=(2\pi\xi_j+(1-\tau)D_{x_j})W_\tau(f,g),
\end{equation*}
cf. \eqref{I13} and  \eqref{I14} for $\tau=1/2$.
\begin{proof}
 We use the metaplectic operator defined in \eqref{tildeW} for  $\tau=\tau_1$ and Lemma \ref{opbtau} for  $\tau=\tau_2$ to write
	\begin{align}\label{primaug}
W_{\tau_1}(Op_{\tau_2}(a)f,g)&=\widetilde{W}_{\tau_1}(Op_{\tau_2}(a)f\otimes \bg)
=\widetilde{W}_{\tau_1}(Op_{\tau_2}(\sigma))(f\otimes\bg)\\
&=\mu({\bf A}_{\tau_1})(Op_{\tau_2}(\sigma))(f\otimes\bg).\notag
\end{align}
From now on we split into the two cases ${\tau_2}=1/2$ or ${\tau_2}\not=1/2$. In fact, when ${\tau_2}=1/2$ we can use the covariance property for Weyl operators. Such property is well known and enjoyed only by Weyl operators, we remark that it does not hold for the other $\tau$-pseudodifferential operators, see for instance \cite{Stein} or \cite{Wong} and the recent contribution \cite{deGossonCovarianceTAMS2013}. Namely, for any $\cA\in Sp(d,\bR)$,
$$\mu(\cA)Op_\tau(\sigma)\mu(\cA)^{-1}=Op_\tau(\sigma\circ \cA^{-1})\Leftrightarrow\,\tau=1/2.$$
That is why we need to study the two cases above separately. Let us start with the easy one: ${\tau_2}=1/2$. Using relation \eqref{primaug}, Lemma \ref{opbtau} and the covariance property above,
\begin{align*}W_{\tau_1}(Op_w(a)f,g)&=\mu({\bf A}_{\tau_1})Op_w(\sigma)(f\otimes\bg)\\
&=Op_w(\sigma\circ {\bf A}^{-1}_{\tau_1})\mu({\bf A}_{\tau_1})(f\otimes\bg)\\
&=Op_w(\sigma\circ {\bf A}^{-1}_{\tau_1})W_{\tau_1}(f,g).
\end{align*}
Now, by \eqref{sigma-tensore} and using the inverse matrix \eqref{Aw-tau-inverse} for $\tau={\tau_1}$,
\begin{align*}\sigma\circ{\bf A}^{-1}_{\tau_1}(x,\xi,u,v)&=\sigma({\bf A}^{-1}_{\tau_1}(x,\xi,u,v))\\
&=a(x-{\tau_1}v,\xi+(1-{\tau_1})u)
\end{align*}

 For $\tau_2\not=1/2$, using \eqref{primaug}, 
and the relation (see e.g., \cite[(4.37)]{Elena-book}), written for arbitrary $\tau_1,\tau_2$:
$$ Op_{\tau_1}(a_1)=Op_{\tau_2}(a_2)\Leftrightarrow\widehat{a_2}(\zeta_{1},\zeta_{2})=e^{-2\pi i (\tau_2-\tau_1)\zeta_{1}\zeta_{2}}\widehat{a_1}(\zeta_{1},\zeta_{2})
$$
we infer
\begin{equation}\label{tau-12}
Op_w(\sigma_{1/2})=Op_\tau(\sigma_\tau)\Leftrightarrow  \sigma_{1/2}=\mu(N_\tau^{-T})\sigma_\tau,
\end{equation}
where the symplectic matrix $N_\tau$ is defined in \eqref{ntau}. Hence
\begin{align*}
W_{\tau_1}(Op_{\tau_2}(a)f,g)&=\mu({\bf A}_{\tau_1})(Op_{\tau_2}(\sigma))(f\otimes\bg)= \mu({\bf A}_{\tau_1})Op_w(\mu(N^{-T}_{\tau_2})\sigma)(f\otimes\bg)\\
&= Op_w((\mu(N^{-T}_{\tau_2})\sigma)\circ {\bf A}^{-1}_{\tau_1})\mu({\bf A}_{\tau_1})(f\otimes\bg)\\
	&= Op_w((\mu(N^{-T}_{\tau_2})\sigma)\circ {\bf A}^{-1}_{\tau_1})W_{\tau_1}(f,g)\\
	&= Op_{\tau_2}(\mu(N^{T}_{\tau_2})[(\mu(N^{-T}_{\tau_2})\sigma)\circ {\bf A}^{-1}_{\tau_1}])W_{\tau_1}(f,g)
\end{align*}
where in the last row  we used \eqref{tau-12} for $\tau=\tau_2$. 

 Note that for $\tau_2=1/2$ we obtain $N_{1/2}=I_{2d\times 2d}$ the identity matrix, and $\mu(N_{1/2})=\mu(N_{1/2}^{-T})=I$, the identity operator, so that the Weyl symbol $b$ in \eqref{simboloWeyl} can be inferred from \eqref{simbolob} when $\tau_2=1/2$.
\end{proof}

Observe that the representation $W_\cA$ in \eqref{WignerA} is a sesquilinear form $L^2(\rd)\times L^2(\rd)\to L^2(\rdd)$.

\begin{prop}
	Consider $\cA\in Sp(2d,\bR)$, then  the representation $W_\cA$ in \eqref{WignerA} is a sesquilinear form from $\cS(\rd)\times\cS(\rd)\to \cS(\rdd)$. 
\end{prop}
\begin{proof}
We recall that the symplectic group is generated by the so-called free symplectic matrices \cite{Birkbis} and thus every metaplectic operator is the product of metaplectic operators associated to free symplectic matrices which reduce to Fourier transforms,  multiplications by chirps,  and linear change of variables. All the operators aforementioned are bounded operators on the Schwartz class. This gives the claim.
\end{proof}

\begin{proposition}[Covariance Property]\label{C4covprop}
Consider $\cA\in Sp(2d,\bR)$ having block decomposition
\begin{equation*} \cA=\left(\begin{array}{cccc}
A_{11} & A_{12}&A_{13}&A_{14}\\
A_{21}&A_{22}& A_{23} &A_{24}\\
A_{31}&A_{32}&A_{33} &A_{34}\\
A_{41}&A_{42}& A_{43} &A_{44}\\
\end{array}\right)
\end{equation*}
with $A_{ij}$, $i,j=1,\dots,4$,  $d\times d$ real matrices. Then the representation $W_\cA$ in \eqref{WignerA} is covariant, namely
\begin{equation}\label{C4covariance}
W_\cA(\pi(z)f)= T_z W_\cA f, \quad  f\in\cS(\rd),\quad  z\in\rdd,
\end{equation}
if and only if $\cA$ is of the form
\begin{equation}\label{A-covariant}
 \cA=\left(\begin{array}{cccc}
A_{11} & I_{d\times d}-A_{11}&A_{13}&A_{13}\\
A_{21}&-A_{21}&I_{d\times d}- A^T_{11} &A^T_{11}\\
A_{31}&-A_{31}&A_{33} &A_{33}\\
A_{41}&-A_{41}& A_{43} &A_{43}\\
\end{array}\right).
\end{equation}
The result does not depend on the choice of the phase factor in the definition of $\mu(\cA)$ and $W_\cA$ in \eqref{WignerA}.
\end{proposition}
\begin{proof}
	For $z=(z_1,z_2)\in\rdd$, we use the intertwining property (see e.g. Formula $(1.10)$ in \cite{Elena-book}) 
	$$\pi(\cA z)=c_\cA \mu(\cA)\pi(z)\mu(\cA)^{-1},\quad z\in\rdd$$
	where $c_\cA$ is a phase factor: $|c_\cA|=1$.  For $z=(z_1,z_2)\in\rdd$ we can write 
	$$W_\cA(\pi(z_1,z_2)f)=\mu(\cA)[\pi(z_1,z_1,z_2,-z_2)(f\otimes \bar{f})]=c_\cA^{-1}\pi(\cA(z_1,z_1,z_2,-z_2))W_\cA f.$$
	Hence $W_\cA$ is covariant if and only if
	\begin{equation}\label{C4carcov}
	\pi(\cA(z_1,z_1,z_2,-z_2))= c_\cA\pi(z_1,z_2,0,0),\quad \forall \,(z_1,z_1)\in\rdd,
	\end{equation}
	where  we used $T_{(z_1,z_2)}=\pi(z_1,z_2,0,0)$. The equality in \eqref{C4carcov} yields $\cA(z_1,z_1,z_2,-z_2)=(z_1,z_2,0,0)$ for every $z_1,z_2\in\rd$. The last equality and the symplectic properties of $\cA$ (see, e.g. \cite[(1.4)--(1.9)]{Elena-book}) give the claim.
\end{proof}

\begin{remark}
(i)	An example of covariant matrix is ${\bf A}_{\tau}$ in \eqref{Aw-tau}, for every $\tau\in [0,1]$ (actually for every $\tau\in\bR$).\\
(ii) The STFT $V_f f$ is not covariant, since the metaplectic matrix ${\bf A_{ST}}$ in \eqref{A-FT} defining $V_f f$ does not satisfy the block matrix decomposition in  \eqref{A-covariant}.
\end{remark}
The special form of the symplectic matrix $\cA$ in \eqref{A-covariant} guarantees the membership of $W_{\cA}$ in the Cohen's class. To compute the corresponding kernel, we begin to write 
$$\cA=\cA  {\bf A}_{1/2}^{-1}  {\bf A}_{1/2}$$
where $ {\bf A}_{1/2}$, $ {\bf A}_{1/2}^{-1}$ are defined as in \eqref{Aw-tau} and \eqref{Aw-tau-inverse} for $\tau=1/2$, so that according to \eqref{mutau-wigner} we have $\mu( {\bf A}_{1/2})(f\otimes \bar{f})=Wf$.
Then 
\begin{equation}\label{aggiunta1}
W_\cA f=\mu(\cA)(f\otimes \bar{f})=\mu(\cA  {\bf A}_{1/2}^{-1})\mu( {\bf A}_{1/2})(f\otimes \bar{f})=\mu(\widetilde{\cA})Wf
\end{equation}
where $\widetilde{\cA}=\cA  {\bf A}_{1/2}^{-1}$ is given by
\begin{equation}\label{aggiunta2}
\widetilde{\cA}=\left(\begin{array}{cc}
I_{2d\times 2d} & B\\
0&I_{2d\times 2d}
\end{array}\right)
\end{equation}
with
\begin{equation}\label{aggiunta3}
B=\left(\begin{array}{cc}
A_{13} & \frac12 I_{d\times d} -A_{11}\\
 \frac12 I_{d\times d} -A^T_{11}&-A_{21}
\end{array}\right).
\end{equation}
In the computation of $\widetilde{\cA}$ we took advantage from the fact that $\widetilde{\cA}$ is symplectic, hence preserving the symplectic form. \par
Applying Proposition \ref{Folland1} to $\widetilde{\cA}\in Sp(2d,\bR)$ with $B$ as in \eqref{aggiunta3} and variables $z=\phas$, $\zeta=(u,v)$, we obtain (modulo phase factors)
\begin{equation}\label{aggiunta4}
\mu(\widetilde{\cA}) F(z)=\intrdd e^{2\pi i S(z,\zeta)}\widehat{F}(\zeta)d\zeta
\end{equation}
where 
\begin{equation}\label{aggiunta5}
 S(z,\zeta)=z\zeta+\frac12 \zeta\cdot B\zeta.
\end{equation}
Hence we conclude from \eqref{aggiunta1}:
\begin{theorem}\label{Thaggiunta1}
Let $\cA\in Sp(2d,\bR)$ be of the form \eqref{A-covariant}. Then
\begin{equation}\label{aggiunta6}
W_\cA f=Wf\ast\sigma
\end{equation}
where 
\begin{equation}\label{aggiunta7}
\sigma(z)=\cF^{-1}_{\zeta\to z}(e^{-\pi i \zeta\cdot B\zeta})\in \cS'(\rdd),
\end{equation}
and  $B$  defined in \eqref{aggiunta3}.
\end{theorem}
By using Proposition \ref{Folland2} we have:
\begin{theorem}\label{Thaggiunta2}
	In the preceding theorem assume  $\det B\not=0$. Then
	\begin{equation}\label{aggiunta8}
	\sigma(z)= e^{\pi i \sharp (B)/4}|\det B| e^{-\pi i \zeta\cdot B^{-1}\zeta},
	\end{equation}
	where $ \sharp (B)$ is the number of positive eigenvalues of $B$ minus the number of negative eigenvalues.
\end{theorem}
\begin{example}\label{exaggiunta1}
Applying the preceding arguments to $\bf{A}_\tau$ in \eqref{Aw-tau} and $W_\tau$ in \eqref{mutau-wigner} we obtain in \eqref{aggiunta2} the expression of $B$:
\begin{equation}\label{aggiunta9}
B= \left(\begin{array}{cc}
0_{d\times d} & (\tau-\frac12)I_{d\times d}\\
 (\tau-\frac12)I_{d\times d}&0_{d\times d}
\end{array}\right),
\end{equation}
that provides in \eqref{aggiunta7} 
\begin{equation}\label{aggiunta10}
\sigma_\tau\phas=\cF^{-1}_{u\to x, v\to \xi} (e^{-2\pi i (\tau-1/2)uv}),
\end{equation}
and we recapture the kernel \eqref{kernel of tau-Wigner distribution} in the Introduction. A more general example is given by $W_\cA$ with $\mu(\cA)$ as in \eqref{metapf2dil}. Under the covariance assumption \eqref{C4covariance}, the matrix $B$ is again anti-diagonal and
\begin{equation}\label{aggiunta11}
\sigma\phas=\cF^{-1}_{u\to x, v\to\xi} (e^{-2\pi i u\cdot Mv})
\end{equation}
for a suitable $M\in GL(d,\bR)$, see \cite{CT2020} and \cite{BCGT2020} for details.
\end{example}
\begin{example}\label{exaggiunta2}
	In Part II of this paper the $\cA$-Wigner representations will play a basic role in the study of Schr\"{o}dinger  equations. In fact, starting from the $\tau$-Wigner distribution of the initial datum, the representation of the evolved solution will require general Cohen's classes, outside those in \eqref{aggiunta10}. Consider here, as example,  
	the free particle equation
	\begin{equation}\label{C26-0} 
	\begin{cases}
	i\partial_t u+\Delta u=0,\\
	u(0,x)=u_0(x),
	\end{cases}
	\end{equation}
	with $(t,x)\in\bR\times\bR^d$, $d\geq1$. For the solution $u(t,x)$ we have from Wigner \cite{Wigner32}
	\begin{equation}\label{C26}
	W u(t,x,\xi)=W u_0(x-4\pi t \xi,\xi).
	\end{equation}
	Looking for a generalization of \eqref{C26} to the case of $W_\tau$, $\tau\in (0,1)$, we may obtain by a direct computation 
	\begin{equation}\label{C27}
	W_{\tau,t} u(t,x,\xi)=W_{\tau,t} u_0(x-4\pi t \xi,\xi),
	\end{equation}
	where the representation $W_{\tau,t}$ is of Cohen class:
	\begin{equation}\label{C28}
	W_{\tau,t}f=Wf\ast\sigma_{\tau,t},
	\end{equation}
	\begin{equation}\label{C29}
	\sigma_{\tau,t}(x,\xi)=\sigma_\tau(x+4\pi t \xi, \xi),
	\end{equation}
	with $\sigma_\tau$ defined in \eqref{kernel of tau-Wigner distribution}. We may write $W_{\tau,t}$ in the form of an $\cA$-Wigner representation, with $\cA$ easily computed and 
	\begin{equation}\label{C30}
	\mu(\cA) F(x,\xi)= \intrd e^{-2\pi i (y\xi+2\pi t(1-2\tau)y^2)}F(x+\tau y,x-(1-\tau)y)\,dy=\cF_2 \mathcal{M}_{\tau,t} \mathfrak{T}_{\tau} F\phas,
	\end{equation}
	where $ \mathcal{M}_{\tau,t}$ is the operator of multiplication by  the chirp $e^{2\pi i t (1-2\tau)y^2}$ and $\mathfrak{T}_{\tau}$ as in Definition \ref{def2.2}.
\end{example}
\section{Almost-diagonalization and wave front sets}\label{S5}
In this section we first study the action of $\tau$-Wigner representations on Weyl operators, then we introduce the $\tau$-Wigner wave front set and provide the almost diagonalization results. Finally, we compare this new wave front set with the classical H{\"o}rmander's global wave front set.
\subsection{Weyl operators and $\tau$-Wigner representations} We first reset Theorem \ref{Tmain} in the frame of the spaces $\cM^p_s$. 
 Also, we want to extend Proposition \ref{P1} to recapture Theorem \ref{T1} in the  more general symbol class $M^{\infty,1}_{1\otimes v_s}(\rdd)$, containing $S^0_{0,0}(\rdd)$. For a symbol  $a\in M^{\infty,1}_{1\otimes v_s}(\rdd)$, $s\geq 0$, we begin to define, for  $r,y,\rho,\eta\in\rd$,
\begin{equation}\label{sigma-sigmatilde}
\sigma(r,y,\rho,\eta) =a(r,\rho)\otimes  1_{(y,\eta)},\quad \tilde{\sigma}(r,y,\rho,\eta) = 1_{(r,\rho)}\otimes a(y,-\eta),
\end{equation}
and, for $\tau\in [0,1]$, ${\bf A}_{\tau}^{-1}$ as in \eqref{Aw-tau-inverse}
\begin{align}
b(x,\xi,u,v)&=(\sigma\circ {\bf A}_{\tau}^{-1})(x,\xi,u,v),\label{A1bis}\\
\tilde{b}(x,\xi,u,v)&=(\tilde{\sigma}\circ {\bf A}_{\tau}^{-1})(x,\xi,u,v),\label{A2bis}\\
c(x,\xi,u,v)&=b(x,\xi,u,v)\tilde{b}(x,\xi,u,v)\label{A3bis}.
\end{align}
In particular, for $\tau=1/2$, we infer
\begin{align}
b(x,\xi,u,v)&=a(x-\frac{v}{2}, \xi+ \frac{u}{2}), \quad \tilde{b}(x,\xi,u,v)=\bar{a}(x+\frac{v}{2},\xi-\frac{u}{2}),\notag\\
c(x,\xi,u,v)&=a(x-\frac{v}{2},\xi+\frac{u}{2})\bar{a}(x+\frac{v}{2},\xi-\frac{u}{2})\label{A3}.
\end{align}

We need to show that the symbols above are in $M^{\infty,1}_{1\otimes v_s}(\bR^{4d})$. This is done in the lemma below.
\begin{lemma}\label{S5A1} Assume  $a\in M^{\infty,1}_{1\otimes v_s}(\rdd)$, $s\geq 0$, $\tau\in [0,1]$. Then,\\
	(i)   The symbol $b =\sigma\circ {\bf A}_{\tau}^{-1}$ in \eqref{A1bis} belongs to $M^{\infty,1}_{1\otimes v_s}(\bR^{4d})$.  \\
	(ii)The symbol $\tilde{b} =\tilde{\sigma}\circ {\bf A}_{\tau}^{-1}$ in \eqref{A2bis} belongs to $M^{\infty,1}_{1\otimes v_s}(\bR^{4d})$.\\
	(iii) The product symbol $c=b\tilde{b}$ in \eqref{A3bis} is in  $M^{\infty,1}_{1\otimes v_s}(\bR^{4d})$.
	
\end{lemma}
\begin{proof}
	$(i)$ Let us show that $\sigma \in M^{\infty,1}_{1\otimes v_s}(\bR^{4d})$.
	First, the constant function $1_{(y,\eta)}$ is in $M^{\infty,1}_{1\otimes v_s}(\rdd)$. In fact, taking a non-zero window $G\in \cS(\rdd)$, 
	$$V_{G} 1_{(y,\eta)}(u_1,u_2)=\cF(T_{u_1}G)(u_2)=M_{-u_1}\widehat{G}(u_2)$$
	and 
	$$\|1_{(y,\eta)}\|_{M^{\infty,1}_{1\otimes v_s}}\asymp \|V_{G} 1_{(y,\eta)}\|_{L^{\infty,1}_{1\otimes v_s}}=\|\widehat{G}\|_{L^1_{v_s}}<\infty$$
		since $\cS(\rdd)\hookrightarrow L^1_{v_s}(\rdd)$ for every $s\geq 0$.
Second, observe that $a\otimes 1_{(y,\eta)}\in  M^{\infty,1}_{1\otimes v_s}(\bR^{4d})$, as can be easily checked by taking the window function $G_1(r,\rho)\otimes G_2(y,\eta)\in S(\bR^{4d})$  for any $G_1,G_2\in \cS(\rdd)\setminus\{0\}$,  noting that
$$V_{G_1\otimes G_2}(a\otimes 1_{(y,\eta)} )(z_1,z_2,\zeta_1,\zeta_2)=V_{G_1}a(z_1,\zeta_1)V_{G_2} 1_{(y,\eta)}(z_2,\zeta_2),\quad z_1,z_2,\zeta_1,\zeta_2\in\rdd$$
and $v_s(\zeta_1,\zeta_2)\leq v_s(\zeta_1)v_s(\zeta_2)$.

We  now focus on $b=\sigma\circ {\bf A}_{\tau}^{-1}$. For any $\tau\in [0,1]$ we have $b\in M^{\infty,1}_{1\otimes v_s}(\bR^{4d})$. In fact, it is well known that affine transformations leave the modulation spaces $M^{p,q}$ invariant. The first result for the Sj\"ostrand class $M^{\infty,1}(\rdd)$ was shown by Sj\"ostrand himself in \cite{Sjostrand1}. More general contributions involving all modulation spaces are contained in \cite{kasso07} and \cite{RSTT2011}. However, there is nothing in the framework of weighted modulation spaces. That is why we will prove the previous statement.  For any fixed non-zero $\Phi$ in $\cS(\bR^{4d})$, 
an easy computation shows that, for every $z,\zeta\in\bR^{4d}$, writing $\Phi_{{\bf A}_{\tau}}:= \Phi\circ {\bf A}_{\tau}$,
\begin{align*}V_{\Phi} b (z,\zeta) &= V_{\Phi} (\sigma\circ {\bf A}_{\tau}^{-1}) (z,\zeta)=|\det {\bf A}_{\tau}|\, V_{\Phi_{{\bf A}_{\tau}}} \sigma ({\bf A}_{\tau}^{-1}z,{\bf A}_{\tau}^{-1}\zeta)\\
&=  V_{\Phi_{{\bf A}_{\tau}}} \sigma({\bf A}_{\tau}^{-1}z,{\bf A}_{\tau}^{-1}\zeta)
\end{align*}
since  ${\bf A}_{\tau}$ is a symplectic matrix.  Now
\begin{align*}
\|b\|_{M^{\infty,1}_{1\otimes v_s}}&\asymp \|V_{\Phi} b \|_{L^{\infty,1}_{1\otimes v_s}} = \int_{\bR^{4d}}\sup_{z\in\bR^{4d}}|V_{\Phi_{{\bf A}_{\tau}}} \sigma({\bf A}_{\tau}^{-1}z,{\bf A}_{\tau}^{-1}\zeta)|v_s(\zeta) d\zeta\\
&= \int_{\bR^{4d}}\sup_{z\in\bR^{4d}}|V_{\Phi_{{\bf A}_{\tau}}} \sigma(z,\zeta)|v_s({\bf A}_{\tau}\zeta) d\zeta\\
&\asymp \int_{\bR^{4d}}\sup_{z\in\bR^{4d}}|V_{\Phi_{{\bf A}_{\tau}}} \sigma(z,\zeta)|v_s(\zeta) d\zeta\asymp\|\sigma\|_{M^{\infty,1}_{1\otimes v_s}}<\infty
\end{align*}
since $|{\bf A}_{\tau}\zeta|\asymp |\zeta|$. \par 
$(ii)$ One can show that $\tilde{b}\in M^{\infty,1}_{1\otimes_{v_s}}(\bR^{4d})$ by using a similar pattern as in the previous stage (i).\par
$(iii)$ We use the product properties for modulation spaces (see \cite[Proposition 2.4.23]{Elena-book}) to infer, for $c=b\tilde{b}$,
$$\|c\|_{M^{\infty,1}_{1\otimes v_s}}\lesssim \|b\|_{M^{\infty,1}_{1\otimes v_s}} \|\tilde{b}\|_{M^{\infty,1}_{1\otimes v_s}}.$$
This concludes the proof.
\end{proof}

Let us recall the following boundedness result, see \cite[Theorem 4.4.15]{Elena-book} (cf. also the early work by Gr\"{o}chenig and Heil  \cite{GH99}).

\begin{lemma}\label{S5A2} If $\sigma\in M^{\infty,1}_{1\otimes v_s}(\bR^{2n})$ then $Op_w(\sigma):\cS(\bR^n)\to \cS'(\bR^n)$ extends to a bounded operator on $\cM^p_{v_s}(\bR^n)$, $1\leq p\leq \infty$, $s\geq 0$.
\end{lemma}

This result is actually valid for  much more general modulation spaces, cf. \cite{Elena-book}. In the sequel the dimension $n$ will be fixed as $d$ or $2d$. It will be also convenient to recall from Corollary \ref{C3.14bis} the implication for $1\leq p\leq\infty$, $s\geq 0$,
\begin{equation}\label{CWigner3.14}
f,g\in M^p_{v_s}(\rd) \Rightarrow W_\tau f,\,W_\tau(f,g)\in M^p_{v_s} (\rdd).
\end{equation}

\begin{theorem}\label{main5.1}
	Consider $\tau\in [0,1]$, $s\geq0$, a symbol $a\in  M^{\infty,1}_{1\otimes v_s}(\rdd)$, functions $f,g\in \cM^p_{v_s}(\rd)$, $1\leq p\leq \infty$. Then, for $b,\tilde{b}, c$ as in \eqref{A1bis}, \eqref{A2bis}, \eqref{A3bis} respectively, we have the following identities in $\cM^p_{v_s}(\rdd)$: 
	\begin{align}
W_\tau(Op_w(a)f,g)&=Op_w(b)W_\tau(f,g),\label{A4}\\
	W_\tau(f, Op_w(a)g)&= Op_w(\tilde{b}) W_\tau(f,g),\label{A5}\\
	W_\tau(Op_w(a)f) &=Op_w(c)W_\tau f. \label{A6}
	\end{align}
\end{theorem}
\begin{proof}
	From Lemma \ref{S5A2} we deduce that $Op_w(a)f\in \cM^p_{v_s}(\rd)$ for any $f\in \cM^p_{v_s}(\rd)$. Then the (cross-)$\tau$-Wigner distributions $W_\tau f$, $W_\tau(f,g)$, $W_\tau(Op_w(a)f,g)$, $W_\tau(f,Op_w(a)g)$, and $W_\tau(Op_w(a)f)$ are in $\cM^p_{v_s}(\rdd)$ in view of the implication in  \eqref{CWigner3.14}. Furthermore, $$Op_w(b)W_\tau(f,g), Op_w(\tilde{b})W_\tau(f,g), Op_w(c)W_\tau f$$
	 in the right-hand side of \eqref{A4}, \eqref{A5}, \eqref{A6} belong to $\cM^p_{v_s}(\rdd)$ in view of Lemmas \ref{S5A1} and \ref{S5A2}.\par
	  Since the Schwartz class $\cS$ in dense in $\cM^p_{v_s}$, by Lemmas \ref{S5A1} and \ref{S5A2} it will be sufficient to prove the identities \eqref{A4}, \eqref{A5} and \eqref{A6} for $f,g\in\cS(\rd)$.  Now, \eqref{A4} is already proved, cf. \eqref{casoWeyl} and \eqref{simboloWeyl}. The equality in  \eqref{A5} is obtained arguing as in the proof of Theorem \ref{Tmain}. Namely, we write 
	  \begin{equation}\label{A7}
	  W_\tau(f,Op_w(a)g)=\widetilde{W}_{\tau}(f\otimes \overline{Op_w(a)g}),
	  \end{equation}
	  where $\widetilde{W}_{\tau}$ is defined in \eqref{tildeW}. Observe that
	  \begin{equation}\label{A7bis}
	  \overline{Op_w(a)g}=Op_w(a^{\ast})\bar{g},
	  \end{equation}
	  where $a^{\ast}(y,\eta)=\bar{a} (y,-\eta)$. 
	  Then by \eqref{A7} 
	  \begin{align}\label{A11}
	 \widetilde{W}_{\tau}(f\otimes \overline{Op_w(a)g})&= \widetilde{W}_{\tau}(f\otimes Op_w(a^{\ast})\bar{g})
	  =\widetilde{W}_{\tau}(Op_w(\tilde{\sigma})(f\otimes \bar{g}))\\ &=\mu({\bf A}_{\tau}) Op_w(\tilde{\sigma})(f\otimes \bar{g}).\notag
	  \end{align}
	  By the covariance property and using the inverse matrix in \eqref{Aw-tau-inverse} we conclude
	   \begin{equation}\label{A12}
	  {W}_{\tau}(f, {Op_w(a)g})=Op_w(\tilde{\sigma}\circ {\bf A}^{-1}_{\tau})W_{\tau}(f,g)=Op_w(\tilde{b})W_{\tau}(f,g),
	  \end{equation}
	  where $\tilde{b}$ is defined in \eqref{A2bis}.
	  Hence \eqref{A5} is proved. As for \eqref{A6}, we may apply repeatedly \eqref{A4} and \eqref{A5}, obtaining
	  \begin{align}\label{A14}
	 W_{\tau}(Op_w(a)f)&= W_{\tau}(Op_w(a)f,Op_w(a)f)=Op_w(\tilde{b})Op_w(b)W_{\tau}(f,g)\\
	 &=Op_w(b)Op_w(\tilde{b})W_{\tau}(f,g).\notag
	  \end{align}
	  So $Op_w(b)$ and $Op_w(\tilde{b})$ commute. It is not clear whether the symbol of their product is simply given by the product of the respective symbols. Concerning this, we may argue as before to obtain
	   \begin{equation}\label{A15}
	  W_{\tau}(Op_w(a)f)=\widetilde{ W}_{\tau}(Op_w(a)f\otimes \overline{Op_w(a)f})=\widetilde{ W}_{\tau}(Op_w(\lambda)(f\otimes \bar{f})),
	  \end{equation} 
	  where now
	  \begin{equation}\label{A16}
	  \lambda(r,y,\rho,\eta)=a(r,\rho)\otimes\bar{a} (y,-\eta).
	  \end{equation}
	  Since
	   \begin{equation}\label{A17}
	  \widetilde{ W}_{\tau}(Op_w(\lambda)(f\otimes \bar{g}))=\mu({\bf A}_{\tau})Op_w(\lambda)(f\otimes \bar{g})=Op_w(\lambda\circ {\bf A}_{\tau}^{-1} )W_{\tau}(f,g)
	  \end{equation} 
	  and
	 \begin{equation}\label{A18}
	 (\lambda\circ {\bf A}_{\tau}^{-1})(x,\xi,u,v)=c(x,\xi,u,v)
	 \end{equation}  
	 with $c$ as in \eqref{A3bis}, we conclude $W_{\tau}(Op_w(\lambda)f)=Op_w(c)W_{\tau}f$. The proof is completed.
\end{proof}
\subsection{$\tau$-Wigner wave front set}
We can now prove Theorems \ref{T2} and \ref{T3} in the Introduction. As before, we shall argue in the more general setting of the $\tau$-Wigner representations and extend Definition \ref{D1} as follows. 
\begin{definition}\label{A4bis}
	Let $f\in\lrd$, $0<\tau<1$. We define $W F_\tau(f)$, the $\tau$-Wigner wave front set of $f$, by setting  $z_0\notin WF_\tau(f)$, $z_0\in\rdd\setminus\{0\}$, if there exists a  conic open neighbourhood  $\Gamma_{z_0}\subset\rdd$ of $z_0$ such that for every integer $N\geq 0$
	\begin{equation*}
	\int_{\Gamma_{z_0}} |z|^{2N}|W_\tau f(z)|^2\,dz<\infty.
	\end{equation*} 
	\end{definition}
We shall limit to consider symbols in the smooth 
H\"{o}rmander class $S^0_{0,0}(\rdd)$,  the intersection of the modulation spaces:
$$S^0_{0,0}(\rdd)= \bigcap_{s\geq 0} M^\infty_{1\otimes  v_s}(\rdd)=\bigcap_{s\geq 0} M^{\infty,1}_{1\otimes v_s}(\rdd),$$ 
cf.  \cite{BC2021,GR}. Observe that Lemmas \ref{S5A1} and \ref{S5A2} apply obviously to this class. As for the functional frame, in the statement of Theorem \ref{T2} we may refer to bounded operators on $\cM^p_{v_s}$, $1\leq p\leq\infty$, $s\geq 0$, whereas in Theorem \ref{T3} we shall use the preceding Definition \ref{A4bis}. From  Theorem \ref{main5.1} we obtain
\begin{equation}\label{A19}
W(Op_w(a)f)(z)=\intrdd k(z,w)Wf(w)\,dw
\end{equation}
where 
\begin{equation}\label{A.20}
k(z,w)=\intrdd e^{2\pi i (z-w)\zeta}c\left(\frac{z+w}{2},\zeta\right) \,d\zeta,
\end{equation}
with $c$ defined as in \eqref{A3bis}, $z=(x,\xi)$, $\zeta =(u,v)$, $w=(y,\eta)$. 

We shall apply for $n=2d$ the following results, valid in any dimension $n$.
\begin{lemma}\label{LemmaA5}
Let $c(z,\zeta)$ be a symbol in $S^0_{0,0}(\bR^{2n})$, $z,\zeta\in\bR^n$. If $k$ denotes the kernel of $Op_w(c)$, then for any integer $N\geq 0$,
\begin{equation}\label{A.21}
k_N:= \la z-w\ra^{2N} k(z,w)
\end{equation}
is the kernel of an operator $Op_w(c_N)$ with $c_N\in S^0_{0,0}(\bR^{2n})$. Moreover, assume $\chi,\f\in\cC^\infty(\bR^n)$ with bounded derivatives of any order and support in two disjoint open cones in $\bR^n\setminus\{0\}$ for  large $z$. Then, for every integer $L\geq 0$ the operator $P_L$ with kernel
\begin{equation}\label{A.22}
\tilde{k}_L(z,w)=\chi(z)\la z\ra^{L}k(z,w)\f(w)
\end{equation} 
is bounded on $L^2(\bR^n)$.
\end{lemma}
\begin{proof}
	We have
	\begin{equation*}
	k_N(z,w)=\int e^{2\pi i (z-w)\zeta}\la z-w\ra^N c\left(\frac{z+w}{2},\zeta\right) \,d\zeta.
	\end{equation*}
	Writing 
	$$\la z-w\ra^{2N} e^{2\pi i (z-w)\zeta}=\left(1-\frac{1}{(2\pi)^2}\Delta_{\zeta}\right)^N e^{2\pi i (z-w)\zeta}
	$$
	and integrating by parts we obtain
	\begin{equation}\label{A.23}
	k_N(z,w)= \int e^{2\pi i (z-w)\zeta} c_N\left(\frac{z+w}{2},\zeta\right) \,d\zeta
	\end{equation}
	with 
	\begin{equation}\label{A.24}
c_N= \left(1-\frac{1}{(2\pi)^2}\Delta_{\zeta}\right)^N c.
\end{equation}
Observe that $c_N\in S^0_{0,0}(\bR^{2n})$ and $k_N$ is the kernel of $Op_w(c_N)$. This proves the first part of the lemma. For the second one, using \eqref{A.21} we may write
$$\tilde{k}(z,w)= \chi(z)\la z\ra^{L}\la z-w\ra ^{-2N}k_N(z,w)\f(w)$$
and therefore from \eqref{A.23}
\begin{equation*}
\tilde{k}_L(z,w)= \int e^{2\pi i (z-w)\zeta}    d_{L,N}(z,w,\zeta) \,d\zeta
\end{equation*} 
where
\begin{equation}\label{A.25}
d_{L,N}(z,w,\zeta)= \chi(z)\la z\ra ^{L}\la z-w\ra^{-2N} c_N\left(\frac{z+w}{2},\zeta\right)\f(w),
\end{equation}
with $c_N$ as in \eqref{A.24}. The integer $N$ will be chosen later. The operator $P_L$ with kernel $\tilde{k}_L$ can be regarded as a pseudodifferential operator defined in terms of the amplitude \eqref{A.25}:
\begin{equation*}
P_L f(z)=\int e^{2\pi i (z-w)\zeta} d_{L,N}(z,w,\zeta) f(w)\,dwd\zeta.
\end{equation*}
Now, observe that $d_{L,N}\in S^0_{0,0}(\bR^{3n})$, that is 
\begin{equation}\label{A.27}
|\partial^\a_z\partial^\beta_w\partial^\gamma_\zeta d_{L,N}(z,w,\zeta)|\leq c_{\a,\beta,\gamma}, \quad (z,w,\zeta)\in \bR^{3n}. 
\end{equation}
In fact, for $z\in$ supp$\chi$ and $w\in$ supp $\f$ we have
\begin{equation}\label{A.28}
\la z\ra \lesssim \la z-w\ra,
\end{equation}
so that choosing $N\geq L$ we prove that $d_{L,N}$ in \eqref{A.25} is bounded. Possibly enlarging $N$ and using again \eqref{A.28} we obtain easily the estimates \eqref{A.27} for every $\a,\beta,\gamma$. We may apply to $P_L$ the generalized version of the Calder\'{o}n-Vaillancourt theorem in \cite{calderon-vaillancourt72} where the $L^2$-boundedness was proved for operators defined by amplitudes satisfying the estimates \eqref{A.27} for a suitable finite set of $\a,\beta,\gamma$. This concludes the proof.
\end{proof}

\begin{proof}[Proof of Theorem \ref{T2} and Theorem \ref{T3}]
	Applying the first part of Lemma \ref{LemmaA5} to $c\in S^0_{0,0}(\bR^{4d})$ and  $k(z,w)$ in \eqref{A.20}, we deduce that $\la z-w\ra ^{2N} k(z,w)$ is the kernel of an operator bounded on $\cM^p_{v_s}(\rdd)$, $1\leq p\leq\infty$, $s\geq 0$, in particular on $\lrdd$.
	
	Hence Theorem \ref{T2} is proved. To obtain the inclusion \eqref{I18bis} in Theorem \ref{T3}, for $f\in\lrd$, we assume $z_0\notin W F_\tau(f)$, $z_0\not=0$, and prove $z_0\notin WF_\tau(Op_w(a)f)$. In view of Definition \ref{A4bis}, the assumption $z_0\notin WF_\tau (f)$  means that there exists an open conic neighbourhood $\Gamma_{z_0}\subset \rdd$ of $z_0$ such that for every integer $N\geq 0$
	\begin{equation}\label{A.29}
	\int_{\Gamma_{z_0}} \la z\ra ^{2N} |W_\tau f(z)|^2\, dz<\infty,
	\end{equation}
	and we want to prove that, possibly shrinking $\Gamma_{z_0}$ to $\Gamma_{z_0}'$, we have, for every $N\geq 0$,
	 \begin{equation}\label{A.30}
	 I=\int_{\Gamma'_{z_0}} \la z\ra ^{2N} |W_\tau (Op_w(a)f)(z)|^2\, dz<\infty.
	 \end{equation}
	 To this end, take first an open conic neighbourhood $\Lambda_{z_0}$ with $\Lambda_{z_0}$$\subset\subset \Gamma_{z_0}$ (we mean the closure of $\Lambda_{z_0}\cap \mathbb{S}^{2d-1}$ is included  in $\Gamma_{z_0}\cap \mathbb{S}^{2d-1}$).
	 Then, consider $\psi\in\cC^\infty(\rdd)$, homogeneous of degree $0$, with $0\leq \psi(z)\leq 1$, $\psi(z)=1$ in $\Lambda_{z_0}$ and supp $\psi\subset \Gamma_{z_0}$ for large $|z|$. Also, apply to $W_\tau (Op_w(a)f)$ the identity \eqref{A6} in Theorem \ref{main5.1} and write in \eqref{A.30} 
	 $$W_\tau (Op_w(a)f)= Op_w(c) W_\tau f,$$
	 with $c$ as in \eqref{A3bis}. We may therefore estimate the integral in \eqref{A.30}
	 $$I\lesssim I_1+ I_2$$
	 with 
	 \begin{equation*}
	 I_1=\int_{\Gamma'_{z_0}} |\la z\ra ^{N}  [Op_w(c)(\psi W_\tau f)](z)|^2\, dz,\quad I_2=\int_{\Gamma'_{z_0}} |\la z\ra ^{N}  [Op_w(c)((1-\psi) W_\tau f)](z)|^2\, dz.
	 \end{equation*}
	 We estimate $I_1$ as follows:
	 $$I_1\leq \|Op_w(c)(\psi W_\tau f)\|^2_{L^2_{v_N}} \leq \|\psi W_\tau f\|^2_{L^2_{v_N}}\lesssim \int_{\Gamma_{z_0}} \la z\ra ^{2N} |W_\tau f|^2(z)\,dz<\infty,$$
	 where we used Lemma \ref{S5A2} and the assumption \eqref{A.29}.
	 
	 To estimate $I_2$ we consider open conic neighbourhoods $\Lambda'_{z_0}, \Gamma'_{z_0}$ of $z_0$, so that $\Gamma'_{z_0}\subset\subset \Lambda'_{z_0}\subset\subset \Lambda_{z_0}\subset\subset \Gamma_{z_0}$. Then we introduce another cut-off function $\chi\in\cC^\infty(\rdd)$, homogeneous of degree $0$, with $0\leq \chi(z)\leq 1$, $\chi(z)=1$ in $\Gamma'_{z_0}$ and with supp $\chi\subset \Lambda'_{z_0}$ for large $|z|$. Note that $\chi$ and $\f=1-\psi$ satisfy the assumptions of the second part of Lemma \ref{LemmaA5}. Now we estimate
	 $$I_2\lesssim \intrdd |\chi(z)\la z\ra^N Op_w(c)[(1-\psi)W_\tau f] |^2 dz.$$
	 The kernel of the operator acting on $W_\tau f$ is of the form \eqref{A.22}, and writing 
	 $$P_N(W_\tau f) =\chi(z)\la z\ra^N Op_w(c)[(1-\psi)W_\tau f] $$
	 we conclude from Lemma \ref{LemmaA5}
	 $$I_2 \leq \|P_N(W_\tau f)\|^2_{2}\leq \|W_\tau f\|_2^2=\|f\|_2^4<\infty.$$
	 For $\tau=1/2$ we obtain in particular Theorem \ref{T3}.
\end{proof}
\begin{theorem}\label{teoZ1} If $f\in\lrd$ and $W F_\tau (f) =\emptyset$, for some $\tau\in (0,1)$, then $f\in\cS(\rd)$.  
\end{theorem}
To prove the above issue we will need the following preliminary result.
\begin{lemma}\label{lemmaZ2}  For $\tau\in [0,1]$ and $g\in\cS(\rd)\setminus\{0\}$, define  $\Psi_\tau =\cI W_\tau g$. Then for every $f\in\lrd$: 
\begin{equation}\label{Z1}
	|V_g f|^2=\Psi_\tau \ast W_\tau f
\end{equation}
\end{lemma}
\begin{proof}
We apply Lemma \ref{lem:STFT of tauWig} with $g=\f_1=\f_2$ and $\zeta=0$. By viewing $V_{\Phi_\tau} W_\tau f(z,0)$ as the convolution $\Psi_\tau \ast W_\tau f(z)$ we obtain \eqref{Z1}. 
\end{proof}
\begin{proof}[Proof of Theorem \ref{teoZ1}]
Assume $W F_\tau (f)= \emptyset$. Then, from the compactness of the sphere $\mathbb{S}^{2d-1}$, we have for every $N$  
\begin{equation}\label{Z2}
\|W_\tau f\|_{L^2_{v_N}}^2=\int_{\rdd} \la z\ra^{2N} |W_\tau f(z)|^2\,dz<\infty.
\end{equation}
Let us prove that the validity of \eqref{Z2} for every $N$ implies for all $M$
\begin{equation}\label{Z3}
|V_g f|\lesssim \la z\ra^{-M}, \quad z\in\rdd,
\end{equation}
for every fixed $g\in\cS(\rd)$, say $g$ the Gaussian; that is, in terms of modulation spaces, $f\in M^\infty_{v_M}(\rd)$. Since $\cap_{M\geq 0} M^\infty _{v_M} =\cS(\rd)$ (cf. \cite[(2.28)]{Elena-book}), we shall conclude $f\in\cS(\rd)$. 

To deduce \eqref{Z3} we apply Lemma \ref{lemmaZ2}:
\begin{equation*}
|V_g f|^2 \leq \intrdd |\Psi_\tau (z-\zeta)| | W_\tau f (\zeta)|\,d\zeta\lesssim \intrdd \la z-\zeta\ra ^{-2M} |W_\tau f(\zeta)|\, d\zeta,
\end{equation*}
for every $M\geq 0$, since $\Psi_\tau\in\cS(\rdd)$. By Peetre's inequality
$$\la z-\zeta\ra ^{-2M}\lesssim \la z\ra ^{-2M}\la \zeta\ra^{2M},
$$
hence by Schwartz' inequality
$$ |V_g f|^2 \lesssim \la z\ra ^{-2M} \intrdd \la \zeta\ra ^{-2d-1} \la \zeta\ra ^{2M+2d+1} |W_\tau f(\zeta)|\,d\zeta\lesssim \la z\ra^{-2M}\|W_\tau f\|_{L^2_{v_{2M+2d+1}}}$$
and \eqref{Z3} follows from \eqref{Z2}.
\end{proof}
\subsection{Comparison with the H{\"o}rmander's global wave front set}
The Wigner wave front certainly deserves a more detailed study. Here we shall limit to a comparison with the wave front set $WF_G$ introduced by H{\"o}rmander \cite{hormanderglobalwfs91} as global version of the standard microlocal wave front, cf.  \cite{HormanderI}. Under the name of Gabor wave front set, and other different names, $W F_G$ has recently had several applications, see \cite{RT2021} for a survey.

Following the notation and  equivalent definition in \cite{RWwavefrontset}, we recall that $z_0=(x_0,\xi_0)\notin  WF_G(f)$, for $f\in\lrd$, $z_0\not=0$, if there exists an open conic neighbourhood $\Gamma_{z_0}\subset \rdd$  such that for every $N\geq 0$
\begin{equation}\label{F1}
\int_{\Gamma_{z_0}} |z|^{2N}|V_g f(z)|^2 dz<\infty,
\end{equation}
where $V_gf$ is the STFT of $f$ defined in \eqref{stftdef}. The definition of $WF_G(f)$ does not depend on the choice of the window $g\in\cS(\rd)\setminus\{0\}$. 

\begin{theorem}\label{theoremZ3}
For all $f\in\lrd$ and $\tau\in (0,1)$ we have 
\begin{equation}\label{Z4}
W F_G (f)\subset WF_\tau (f).
\end{equation}
\end{theorem}
\begin{proof}
	Assume $z_0=(x_0,\xi_0)\notin W F_\tau (f)$, that is the estimate in Definition \ref{A4bis} are satisfied for every $N$ in a suitable conic neighbourhood $\Gamma_{z_0}$. Let us prove that \eqref{F1} is valid for every $N$, by shrinking  $\Gamma_{z_0}$ to $\Gamma_{z_0}'\subset\subset  \Gamma_{z_0}$, namely:
	\begin{equation}\label{Z5}
	\int_{\Gamma_{z_0}'}\la z\ra^{2N}|V_g f|^2(z)\,dz<\infty.
	\end{equation}
	It will be sufficient to estimate for every $M\geq 0$
	\begin{equation}\label{Z6}
	|V_g f|^2(z)\,dz\lesssim \la z\ra^{2M}\quad \mbox{for}\, z\in \Gamma_{z_0}'.
	\end{equation}
	Applying Lemma \ref{lemmaZ2} and arguing as in the proof of Theorem \ref{teoZ1} we have for every $Q\geq 0$ 
	$$|V_g f|^2(z)\leq I_1+ I_2$$
	with 
	$$I_1=\int_{\Gamma_{z_0}} \la z-\zeta\ra^{-2Q}|W_\tau f(\zeta)|\,d\zeta,\quad I_2=\int_{\rdd\setminus \Gamma_{z_0}} \la z-\zeta\ra^{-2Q}|W_\tau f(\zeta)|\,d\zeta.$$
	Since the restriction of $W_\tau f(\zeta)$ to $\Gamma_{z_0}$ satisfies the estimates \eqref{Z2} in $\rdd$, for $I_1$ we may argue as in the proof of Theorem \ref{teoZ1} and deduce  the estimates \eqref{Z6} in $\rdd$. As for $I_2$, we note that for $z\in \Gamma_{z_0}'$ and $\zeta\in \rdd\setminus \Gamma_{z_0}$ we have
	$$ \la \zeta\ra\lesssim\la z-\zeta\ra, \quad \la z\ra\lesssim \la z-\zeta\ra.$$
	Hence by taking $Q=M+d$
	$$I_2\lesssim \la z\ra ^{-2M}\intrdd \la \zeta\ra ^{-2d}|W_\tau f(\zeta)\,d\zeta,$$
	for $z\in \Gamma_{z_0}'$. Since $f\in\lrd$ implies $W_\tau f\in\lrdd$, \eqref{Z3} follows and Theorem \ref{theoremZ3} is proved.
\end{proof}

The similarity of \eqref{F1} and \eqref{I17} leads naturally to ask whether the wave front sets $WF$ in Definition \ref{D1} and $WF_G$ coincide. This can be easily tested on examples in dimension $d=1$.  Consider first $f\in L^2(\bR)$ with compact support, say supp $f\subset [a,b]$, with $a,b\in\bR$, $a<b$. From the standard support property for Wigner transform, we have that supp $Wf$ is included in the strip $\{(x,\xi)\in\bR^2, a\leq x\leq b\}$, hence $WF(f)\subset \{(x,\xi), x=0\}$ in view of Definition \ref{D1}. Same inclusion is valid for $WF_G(f)$, see for example Section $6.6.4$ in \cite{Elena-book}.  Similarly, we may consider $g\in L^2(\bR)$ with compactly supported $\hat{g}$ and obtain that both wave fronts are included in the $x$ axis $\{(x,\xi), \xi=0\}.$ It is easy to prove that we have the identities
\begin{equation}\label{F2}
WF(f)=WF_G(f),\quad  WF(g)=WF_G(g). 
\end{equation} 
The situation changes drastically if we consider the sum $f+g$. In fact, by linearity \eqref{F1} gives
\begin{equation}\label{F2}
WF_G(f+g)=WF_G(f)\cup  WF_G(g),
\end{equation} 
and  $WF_G(f+g)$ is the union of the axes in $\bR^2$ for suitable $f,g\in\lrd$, see below.
Instead, 
\begin{equation}\label{F3}
W(f+g)=Wf+  Wg +2 {\mathcal Re} W(f,g),
\end{equation} 
and the cross-Wigner term may produce an additional \emph{ghost} part of the wave front, according to the presence of the so-called \emph{ghost frequencies} in Signal Theory \cite{bogetal}. To be definite, fix $f\in\cC^\infty(\bR\setminus\{0\})$, $0\leq f(x)\leq 1$, $f(x)=1$ for $-1/2\leq x< 0$ and $f(x)=0$ for $x\leq -1$ and $x> 0$. For $g\in L^2(\bR)$ we take $g=-2\pi \hf$. Let us test \eqref{I17} for small $\Gamma_{z_0}$ with $z_0$ outside the axes and $N=2$. It will be sufficient to consider $x\xi W(f,g)\phas$. An easy computation by Moyal operators gives
\begin{equation}\label{F5}
8\pi x \xi W(f,g)=W(x Df,g)+ W(f,xDg)+W(xf,Dg) + W(Df,xg).
\end{equation}
Differentiating $f$ in the distributions' framework gives 
\begin{equation}\label{F6}
Df=i\delta+ i f',
\end{equation}
with $f'\in\cC^\infty(\bR)$ with compact support, hence
\begin{equation}\label{F7}
xg=-2\pi x \hf =-\widehat{Df}=-i-ih,
\end{equation}
with $h\in\cS(\bR)$. From the definition of $f,g$ and from \eqref{F6}, \eqref{F7} we deduce that $xf,Dg, xDf, xDg$ belong to $L^2(\bR)$, therefore by the Moyal $L^2$ identity all the terms in the right-hand side of \eqref{F5} are in $L^2(\bR^2)$, but the term $W(Df,xg)$. On the other hand, in view of $\eqref{F6}$, $\eqref{F7}$,
\begin{equation}\label{F8}
W(Df,xg)=W(\delta,1)+ W(\delta,h)+ W(f',1)+W(f',h).
\end{equation}
Since $h,f'\in\cS(\bR)$, then $W(f',h)\in\cS(\bR^2)$, and $W(\delta,h)$, $W(f',1)$ are of rapid decay in the complement of any conic neighbourhood of the axes in $\bR^2$. It remains to consider 
\begin{equation}\label{F9}
W(\delta,1)=\int_{\bR} e^{-2\pi i t\xi }\delta_{x+\frac t2} dt= e^{4\pi i x\xi },
\end{equation}
that provides a non-convergent integral in \eqref{I17}, for any $z_0=(x_0,\xi_0)$, $x_0\not=0$, $\xi_0\not=0$ and any neighbourhood $\Gamma_{z_0}$. Hence $WF(f+g)=\bR^2\setminus\{0\}$, i.e. the ghost wave front invades the whole $\bR^2$. 
\begin{remark}
Similar examples can be given for $W_\tau f$. Actually, for $\tau\not=1/2$ the $\tau$-wave front is not limited to the convex closure of $W F_G ( f)$, in particular $W F_G ( f)$ may consist of a single ray and $W F_\tau (f)$ be a larger cone.\par
In conclusion, let us suggest, without giving details, an alternative approach to the Wigner microlocal analysis. Namely, we may replace the rapid decay in cones expressed by Definitions \ref{D1} and \ref{A4bis} with the \textbf{distributional} rapid decay characterizing the space $(\mathcal{O}'_C)$ of Schwartz (\cite[Chapter 7, Section 5]{schwartz}). According to Example $(VII, 5;1)$  in \cite{schwartz} the chirp function belongs to $(\mathcal{O}'_C)$, hence it is, somehow surprisingly, a distribution of rapid decay. We may extend the argument of Schwartz to all the ghost part of  $W F_\tau (f)$, as suggested by \eqref{F9} and Lemma \ref{lemmaZ2}. In this perspective, ghosts do not exist, so Wigner wave front and H\"{o}rmander global wave front  coincide.
\end{remark}
\section*{Acknowledgements}
The first  author has been supported by the Gruppo Nazionale per l’Analisi Matematica, la Probabilità e le loro Applicazioni (GNAMPA) of the Istituto Nazionale di Alta Matematica (INdAM).

We thank Maurice de Gosson for reading the manuscript and providing  useful comments.

\end{document}